\documentclass[a4paper,12pt,reqno]{amsart}
\usepackage{verbatim}
\usepackage{enumitem}
\usepackage{amssymb}
\usepackage[all]{xy}
\usepackage{mathtools}
\usepackage{color}
\usepackage{tabu}
\usepackage{longtable}
\usepackage{float}
\usepackage{amsthm}
\usepackage{booktabs}
\definecolor{cobalt}{RGB}{61,89,171}
\usepackage[colorlinks,citecolor=cobalt,linkcolor=cobalt,urlcolor=cobalt,pdfpagemode=UseNone,backref = page]{hyperref}

\newcommand{\x}{\times}
\newcommand{\ox}{\otimes}
\newcommand{\RR}{\mathbb{R}}
\newcommand{\la}{\langle}
\newcommand{\ra}{\rangle}

\newcommand{\G}{\Gamma}
\newcommand{\vol}{\mathrm{vol}}

\newcommand{\R}{\mathbb{R}}
\newcommand{\lie}[1]{\mathfrak{#1}}     

\newcommand{\SU}{\mathrm{SU}}

\newcommand{\Gtwo}{\mathrm{G}_2}

\newcommand{\fg}{\mathfrak{g}}
\newcommand{\fh}{\mathfrak{h}}
\newcommand{\fn}{\mathfrak{n}}

\newcommand{\GL}{\mathrm{GL}}

\DeclareMathOperator{\tr}{tr}

\newtheorem{theorem}{Theorem}[section]
\newtheorem*{theorem*}{Theorem}
\newtheorem{lemma}[theorem]{Lemma}
\newtheorem{corollary}[theorem]{Corollary}
\newtheorem{proposition}[theorem]{Proposition}

\newtheorem{example}[theorem]{Example}
\newtheorem{conjecture}[theorem]{Conjecture}

\theoremstyle{remark}
\newtheorem{remark}[theorem]{Remark}

\begin{document}

\title{Purely coclosed $\Gtwo$-structures on nilmanifolds}

\author[G. Bazzoni]{Giovanni Bazzoni}
\address{Dipartimento di Scienza ed Alta Tecnologia, Universit\`a degli Studi dell'Insubria, Via Valleggio 11, 22100, Como, Italy}
\email{giovanni.bazzoni@uninsubria.it}

\author[A. Garv\'{\i}n]{Antonio Garv\'{\i}n}
\address{Departamento de Matem\'atica Aplicada, 
Escuela de Ingenier\'{\i}as industriales, Campus Teatinos, Universidad de M\'alaga, 
Apdo 59. 29080 M\'alaga, Spain}
\email{garvin@uma.es}

\author[V. Mu\~{n}oz]{Vicente Mu\~{n}oz}
\address{Departamento de \'Algebra, Geometr\'{\i}a y Topolog\'{\i}a, Universidad de M\'alaga, 
Campus de Teatinos, s/n, 29071 M\'alaga, Spain}
\email{vicente.munoz@uma.es}


\begin{abstract}
We classify $7$-dimensional nilpotent Lie groups, decomposable or of step at most 4, endowed with left-invariant purely coclosed $\Gtwo$-structures. This is done by going through the list of all $7$-dimensional nilpotent Lie algebras given by Gong \cite{Gong},
providing an example of a left-invariant $3$-form $\varphi$ which is a pure coclosed $\Gtwo$-structure (that is, it satisfies $d*\varphi=0$,
$\varphi \wedge d\varphi=0$) for those nilpotent Lie algebras that admit them; and by showing the impossibility
of having a purely coclosed $\Gtwo$-structure for the rest of them.
\end{abstract}

\maketitle

\section{Introduction}\label{sec:intro}

A $7$-dimensional smooth manifold $M$ admits a $\Gtwo$-structure if the structure group of its frame bundle reduces to the exceptional Lie group $\Gtwo \subset \mathrm{SO}(7)$. Equivalently (see \cite{Bryant}), $M$ admits a $\Gtwo$-structure if and only if it is orientable and spin.
Further, a $\Gtwo$-structure is equivalent to the existence of a positive $3$-form $\varphi$ (see Section \ref{tools} for details), which defines a unique Riemannian metric $g_{\varphi}$ and an orientation $\vol_\varphi$ on $M$. When $\varphi$ is parallel
with respect to the Levi-Civita connection of $g_\varphi$, then the identity component of its holonomy group is contained in $\Gtwo$; Fern\'andez and Gray proved that this happens if and only if $\varphi$ is closed and coclosed \cite{FernandezGray}. In this case, $g_\varphi$ is Ricci-flat. A $\Gtwo$-structure is called {\it closed} if $d\varphi=0$, and {\it coclosed} if $d*_\varphi\varphi=0$, where 
$*_\varphi$ is the Hodge star operator associated to $g_\varphi$ and $\vol_\varphi$. These two classes of $\Gtwo$-structures are very different in nature; for instance, the closed condition is quite restrictive (see the recent survey \cite{Fino-Raffero}), while coclosed $\Gtwo$-structures exist on any closed, oriented spin manifold, since they satisfy an $h$-principle, as proved by Crowley and Nordstr\"om in \cite[Theorem 1.8]{CN}. 

As it is the case for general $\mathrm{G}$-structures, the non-integrability of a $\Gtwo$-structure is governed by its intrisic torsion $\tau$, see \cite{Salamon}. In this particular case,
$\tau$ has four components $\tau_0$, $\tau_1$, $\tau_2$, and $\tau_3$, with $\tau_i\in\Omega^i(M)$, determined by the equations
\[
\left\{
\begin{array}{lcl}
d\varphi & = & \tau_0*_\varphi\varphi + 3~\tau_1\wedge \varphi + *_\varphi\tau_3 \\
d*_\varphi\varphi & = & 4\tau_1 \wedge *_\varphi\varphi + \tau_2\wedge \varphi 
\end{array}
\right.\,;
\]
see \cite[Proposition 1]{Bryant}. According to the vanishing of the various torsion components, one obtains sixteen classes of $\Gtwo$-structures, see \cite{FernandezGray}. We recognize closed $\Gtwo$-structures as those for which $\tau_0=\tau_1=\tau_3$; on the other hand, coclosed $\Gtwo$-structures are characterized by $\tau_1=\tau_2=0$. A $\Gtwo$-structure is of pure type if all the torsion components vanish, but one. Thus closed $\Gtwo$-structures are of pure type, while coclosed $\Gtwo$-structures are not. A $\Gtwo$-structure is locally conformally parallel (see \cite{IPP}) if $\tau_0=\tau_2=\tau_3=0$, that is, if $d\varphi=\tau_1\wedge\varphi$ and $d*_\varphi\varphi=\tau_1\wedge *_\varphi\varphi$. In this case, locally there is a function $f$ such that $e^{f} \varphi$ is a parallel $\Gtwo$-structure. Nearly parallel $\Gtwo$-structures (see \cite{FKMS}) are another important pure class; they are characterized by $\tau_1=\tau_2=\tau_3=0$, i.e.\ $d\varphi=\tau_0*_\varphi\varphi$, where $\tau_0$ is a constant. In this case, the induced metric $g_\varphi$ is Einstein with positive scalar curvature.

In this paper we focus on the last pure class of $\Gtwo$-structures, called {\em purely coclosed $\Gtwo$-structures}; these are 
given by the conditions $\tau_0=\tau_1=\tau_2=0$, that is, $d\varphi=*_\varphi\tau_3$ and $d*_\varphi\varphi=0$, or, 
equivalently \cite{BMR}, by
 $$
d*_\varphi\varphi=0 \quad \mathrm{and} \quad \varphi\wedge d\varphi=0\,;
 $$
clearly, they are a subclass of coclosed $\Gtwo$-structures. The second one is an equality of $7$-forms, hence it imposes a single extra condition. 
It is not clear whether there is an $h$-principle for purely coclosed $\Gtwo$-structures.

A nilmanifold is a compact quotient $M=\G\backslash G$, where $G$ is a connected, simply connected, nilpotent Lie group, and $\Gamma\subset\,G$ is a lattice. By Mal'cev Theorem \cite{Malcev}, a lattice $\Gamma\subset\,G$ exists if and only if the
Lie algebra $\lie{g}$ of $G$ has a basis with respect to which the structure constants are rational numbers. A nilmanifold is parallelizable, hence it is spin for any Riemannian metric.
Thanks to the aforementioned results, every nilmanifold has a coclosed $\Gtwo$-structure.

We are interested on nilmanifolds endowed with {\em left invariant} $\Gtwo$-structures. Since left-invariant differential forms on $\G\backslash G$ are uniquely determined by forms on $\lie{g}$, one can restrict the attention to $7$-dimensional real nilpotent Lie algebras, which have been classified by Gong in \cite{Gong}. Conti and Fern\'andez classified nilpotent Lie groups endowed with a left-invariant closed $\Gtwo$-structure, see \cite{CF}. Nilmanifolds can not carry locally conformally parallel $\Gtwo$-structures. Indeed, by a result of Ivanov, Parton and Piccinni \cite{IPP}, a compact manifold $M$ endowed with a locally conformally parallel $\Gtwo$-structure fibers over the circle with fiber a compact, simply connected 6-manifold, hence $b_1(M)=1$, while the first Betti number of a nilmanifold is at least $2$. Also, non-toral nilmanifolds can not have left-invariant nearly parallel $\Gtwo$-structures. In fact, as we noticed above, the induced metric is Einstein in this case, and this never happens for nilmanifolds, due to a result of Milnor \cite[Theorem 2.4]{Milnor}. As for left-invariant coclosed $\Gtwo$-structures, there is an unpublished classification by Bagaglini \cite{Bagaglini}, which does not seem to be complete. In \cite{BFF}, Bagaglini, Fern\'andez and Fino determined which nilpotent Lie groups admit left-invariant coclosed $\Gtwo$-structures in two cases: when the Lie algebra is decomposable, and when it is $2$-step. The authors also showed that one can always choose the $\Gtwo$-structure in such a way that the induced metric is a nilsoliton. In \cite{Freibert}, Freibert obtained all nilpotent almost-abelian Lie algebras with a coclosed $\Gtwo$-structure. In \cite{BMR}, del Barco, Moroianu and Raffero classified $2$-step nilpotent Lie groups admitting left-invariant purely coclosed invariant $\Gtwo$-structures. Their approach has a theoretical flavor, and does not rely on the classification of $7$-dimensional nilpotent Lie algebras. 

In this paper we study the existence of left-invariant purely coclosed $\Gtwo$-structures on $7$-dimensional decomposable nilpotent Lie groups and on indecomposable nilpotent Lie groups with nilpotency step $\leq 4$. We determine those which admit a left-invariant purely coclosed $\Gtwo$-structures. For this, we go by exhaustion through the list of nilpotent Lie algebras of \cite{Gong}. For each Lie group, in the positive case we provide an explicit example of a left-invariant purely coclosed $\Gtwo$-structure. In the negative case, we show that it is not possible to find such $\Gtwo$-structure by showing that there are suitable obstructions that forbid this to happen. 
The results are summarized in Theorems \ref{thm:decomposable-pure},
\ref{2step-indecomposable-pure}, \ref{3step-indecomposable-pure} and \ref{4step-indecomposable-pure}. In particular, we have the following

\begin{theorem*}\label{thm:main}
Every $7$-dimensional decomposable nilpotent Lie algebra admitting a coclosed $\Gtwo$-structure also admits a purely coclosed one, except for $\fh_3\oplus\RR^4$, where $\fh_3$ is the Heisenberg Lie algebra.
Every $7$-dimensional indecomposable nilpotent Lie algebra of nilpotency step $\leq 4$ admitting a coclosed $\Gtwo$-structure also admits a purely coclosed one.
\end{theorem*}

In future work we shall address the remaining nilpotency steps.

For the ease of the verification, we use SageMath worksheets \cite{Sage}. For each of the Lie algebras there is a single
worksheet. When the Lie algebra admits a left-invariant purely coclosed $\Gtwo$-structure, we provide the
relevant forms, and the list of commands that verify that it is purely coclosed $\Gtwo$-structure. In the negative case,
we provide a commented worksheet with all the steps that show the obstructions. In the former case, we use 
different obstructions depending on the case, we explain in the theoretical part of the paper how to use these worksheets.
The SageMath worksheets can be found at \cite{worksheets}.

\noindent {\bf Acknowledgements.} 
We thank Marisa Fern\'andez, Anna Fino and Alberto Raffero for useful comments and references. We have taken as a departing
point the (unfortunately incomplete) preprint of Leonardo Bagaglini \cite{Bagaglini}. Part of our search of forms has been
shortened by this.
The second and third author were partially supported MINECO (Spain) grant PGC2018-095448-B-I00.

\section{Generalities on $\Gtwo$-structures}\label{tools}

Let $V$ be a $7$-dimensional
vector space. Given a $3$-form $\varphi\in \Lambda^3(V^*)$, we define a symmetric bilinear form $b_\varphi\colon V \x V\to \Lambda^7(V^*)$ by $b_\varphi(x,y)\coloneqq \frac16 \imath_x\varphi \wedge \imath_y\varphi\wedge \varphi$. We have that
$\varphi$ is non-degenerate if $\epsilon(\varphi)\coloneqq(\det(b_\varphi))^{1/9} \neq 0$. Then 
 \begin{equation}\label{eq:metric$7$-dim}
 g_\varphi\coloneqq\epsilon(\varphi)^{-1} b_\varphi
 \end{equation} 
is also a symmetric bilinear form on $V$. If it is positive definite, then $\varphi$ is a
{\em positive} 3-form. By definition this is called a $\Gtwo$-form. Then, there is a $g_\varphi$-orthonormal frame $\{e_1,\ldots, e_7\}$ such that 
\[
\varphi=e^{127} +e^{347} +e^{567} +e^{135} -e^{146} -e^{236} -e^{245}\,, 
\]
where $\{e^1,\ldots, e^7\}$ is the dual coframe and $e^{ij}\coloneqq e^i\wedge e^j$, $e^{ijk}\coloneqq e^i\wedge e^j\wedge e^k$, and so on.

Recall that a $7$-dimensional smooth manifold $M$ is said to admit a $\Gtwo$-structure if there is a reduction of the structure group of its frame bundle from ${\GL}(7,\mathbb{R})$ to the exceptional Lie group $\Gtwo$, which can actually be viewed naturally as a subgroup of $\mathrm{SO}(7)$. 
Thus, a $\Gtwo$-structure determines a Riemannian metric and an orientation on $M$. In fact, the presence of a $\Gtwo$-structure is equivalent to the existence of a $3$-form $\varphi$ (the $\Gtwo$-form) on $M$, which is \emph{positive} on each tangent space $T_{p} M$, $p\in M$.

By \eqref{eq:metric$7$-dim}, a $\Gtwo$-form $\varphi$ induces both an orientation $\vol_\varphi$ and a Riemannian metric $g_{\varphi}$ on $M$, given by
\begin{equation}\label{metric}
6\, g_{\varphi} (X,Y)\, \vol_\varphi= \imath_X\varphi \wedge \imath_Y\varphi \wedge \varphi,
\end{equation}
for vector fields $X$, $Y$ on $M$. Let $*_\varphi$ be the Hodge star operator determined by $g_\varphi$ and $\vol_\varphi$. 
We say that a manifold $M$ has a {\em coclosed $\Gtwo$-structure}
if there is a $\Gtwo$-structure on $M$ such that the $\Gtwo$-form $\varphi$ is coclosed, that is, $d*_\varphi\varphi=0$. The $\Gtwo$-structure is called {\em purely coclosed} if, in addition, $\varphi\wedge d\varphi=0$.

\section{Linear $\SU(3)$-structures}

One way to understand $\Gtwo$-structures on $7$-manifolds is in terms of $\SU(3)$-structures on 6-manifolds. In fact, both structures can be described coherently using spinors, as it was shown in \cite{ACFH}.

Let $V$ be a $6$-dimensional 
vector space. Define the set
\[
\Lambda_0(V^*)=\{\omega \in \Lambda^2(V^*) \mid \omega^3  = 0\}\,.
\]
Given $\omega \in \Lambda_0(V^*)$, we orient $V$ declaring $\omega^3>0$.

For every $\tau\in \Lambda^3 (V^*)$, we have a map $k_\tau\colon V\to \Lambda^5(V^*)$ given by $k_\tau(x)=\imath_x\tau \wedge \tau$, where $\imath_x$ denotes contraction. Recall the natural isomorphism 
$V\ox \Lambda^6(V^*)\to \Lambda^5(V^*)$, given by $(v,o)\mapsto\imath_v o$. 
Its inverse is $\mu\colon \Lambda^5(V^*) \to V\ox \Lambda^6(V^*)$; if we fix a basis $\{v_1,\ldots,v_6\}$ of $V$, with dual basis $\{v^1,\ldots,v^6\}$, then $\mu(\xi)=\sum v_\ell \otimes (v^\ell \wedge \xi)$, for $\xi\in\Lambda^5(V^*)$. Composing $k_\tau$ with $\mu$ we obtain a map 
\[
 K_\tau\coloneqq \mu \circ k_\tau\colon V\to V\ox \Lambda^6(V^*)\,.
\]

In turn, this determines a function $\lambda\colon\Lambda^3(V^*) \to \left(\Lambda^6(V^*)\right)^{\ox 2}$ by
\[
 \lambda (\tau)=\frac16 \tr\big((K_\tau \otimes 1_{\Lambda^6(V^*)} )\circ K_\tau\big) \in \left(\Lambda^6(V^*)\right)^{\ox 2}\,.
\]
Set 
\[
\Lambda_\pm(V^*)\coloneqq\{ \tau\in \Lambda^3(V^*) \mid \pm \lambda(\tau)>0\}\,.
\]
Note that this condition is independent of orientations.

Take $V_e=\la e_1,e_2,e_3,e_4,e_5,e_6 \ra$, and consider copies $V_f=\la f_1,f_2,f_3,f_4,f_5,f_6 \ra$, $V_g=\la g_1,g_2,g_3,g_4,g_5,g_6 \ra$, and $V_h=\la h_1,h_2,h_3,h_4,h_5,h_6 \ra$. Given $\tau\in \Lambda^3(V^*)$, denote by $\tau_e$ (resp.~$\tau_f$, $\tau_g$, $\tau_h$) the corresponding elements in $\Lambda^3(V_e^*)$ (resp.~$\Lambda^3(V_f^*)$, $\Lambda^3(V_g^*)$, $\Lambda^3(V_h^*)$). 
A {\em $(1,2,3)$-shuffle} is a collection $\sigma=\{\{i\},\{j,k\},\{r,s,t\}\}$ which is a permutation of $\{1,\ldots,6\}$
 with $j<k$ and $r<s<t$. The {\em sign} of a $(1,2,3)$-shuffle $\sigma$, $(-1)^\sigma$, is its sign as permutation. Consider the element 
\[
C_{xyz}= \sum_{\sigma} (-1)^\sigma x^i \wedge y^j\wedge y^k \wedge z^r\wedge z^s\wedge z^t\,,
\]
for $x,y,z=e,f,g,h$, where the sum runs over all $(1,2,3)$-shuffles. We have the following:
 
\begin{proposition} \label{prop:Cefgh}
Let $\tau \in \Lambda^3(V^*)$. Then
\[
\tau_e\wedge C_{gfe} \wedge \tau_f\wedge \tau_h \wedge C_{fgh}\wedge \tau_g= -6\lambda(\tau) \, \vol_{efgh}
\]
\end{proposition}
  
\begin{proof}
There is an easy equality 
\[
C_{gfe} \wedge \tau_e=-\sum_{\ell=1}^6 g^\ell\wedge \imath_{f_\ell} \tau_f\wedge\vol_e\,,
\]
from which one has
 \begin{equation}\label{eqn:formula}
 \tau_e \wedge C_{gfe} \wedge\tau_f= -\sum_{\ell=1}^6 g^\ell\wedge k_{\tau_f}(f_\ell)\wedge\vol_e\,.
 \end{equation}
This implies
\begin{align*}
\tau_e\wedge C_{gfe} \wedge \tau_f\wedge \tau_h \wedge C_{fgh}\wedge \tau_g&=\sum_{\ell,m=1}^6g^\ell\wedge k_{\tau_f}(f_\ell)\wedge f^m\wedge k_{\tau_g}(g_m)\wedge\vol_{eh}\\
&=-\sum_{\ell,m=1}^6g^\ell\wedge k_{\tau_g}(g_m) \wedge f^m\wedge k_{\tau_f}(f_\ell)\wedge\vol_{eh}\,.
\end{align*}
Recall that $K_\tau=\mu\circ k_\tau$, so that $K_{\tau_g}(g_m)= \sum g_\ell\otimes g^\ell \wedge k_{\tau_g}(g_m)$. Hence $\kappa_{\ell m}$, the entry $(\ell,m)$ of the matrix of $K_{\tau_g}$, is determined by $\kappa_{\ell m} \vol_g=g^\ell \wedge k_{\tau_g}(g_m)$. 
This implies that the above sum is 
\[
-\sum_{\ell,m=1}^6g^\ell\wedge k_{\tau_g}(g_m) \wedge f^m\wedge k_{\tau_f}(f_\ell)=
- \sum_{\ell,m=1}^6 \kappa_{\ell m}\kappa_{m \ell} \vol_{gf} =
 - \tr(K_\tau^2)\vol_{gf}=-6\lambda(\tau)\vol_{gf}\,,
\]
and the result follows.
\end{proof}

We implement this computation in SageMath as follows:

{\tiny

\noindent \hspace{1mm} \begin{verb}
D.<e1,e2,e3,e4,e5,e6,f1,f2,f3,f4,f5,f6,g1,g2,g3,g4,g5,g6,h1,h2,h3,h4,h5,h6> = GradedCommutativeAlgebra(QQ)
\end{verb}

\noindent \hspace{1mm} \begin{verb}
N=D.cdg_algebra({})
\end{verb}

\noindent \hspace{1mm} \begin{verb}
N.inject_variables()
\end{verb}

\noindent \hspace{1mm} \begin{verb}
psie=-e2*e4*e6+e1*e3*e6+e1*e4*e5+e2*e3*e5
\end{verb}

\noindent \hspace{1mm} \begin{verb}
psif=-f2*f4*f6+f1*f3*f6+f1*f4*f5+f2*f3*f5
\end{verb}

\noindent \hspace{1mm} \begin{verb}
psig=-g2*g4*g6+g1*g3*g6+g1*g4*g5+g2*g3*g5
\end{verb}

\noindent \hspace{1mm} \begin{verb}
psih=-h2*h4*h6+h1*h3*h6+h1*h4*h5+h2*h3*h5
\end{verb}

\noindent \hspace{1mm} \begin{verb}
Cgfe=g1*f2*f3*e4*e5*e6+g1*f2*e3*f4*e5*e6+g1*f2*e3*e4*f5*e6+g1*f2*e3*e4*e5*f6+g1*e2*f3*f4*e5*e6
\end{verb}

\begin{verb}
+g1*e2*f3*e4*f5*e6+g1*e2*f3*e4*e5*f6+g1*e2*e3*f4*f5*e6+g1*e2*e3*f4*e5*f6+g1*e2*e3*e4*f5*f6
\end{verb}

\begin{verb}
+f1*g2*f3*e4*e5*e6+f1*g2*e3*f4*e5*e6+f1*g2*e3*e4*f5*e6+f1*g2*e3*e4*e5*f6+e1*g2*f3*f4*e5*e6
\end{verb}

\begin{verb}
+e1*g2*f3*e4*f5*e6+e1*g2*f3*e4*e5*f6+e1*g2*e3*f4*f5*e6+e1*g2*e3*f4*e5*f6+e1*g2*e3*e4*f5*f6
\end{verb}

\begin{verb}
+f1*f2*g3*e4*e5*e6+f1*e2*g3*f4*e5*e6+f1*e2*g3*e4*f5*e6+f1*e2*g3*e4*e5*f6+e1*f2*g3*f4*e5*e6
\end{verb}

\begin{verb}
+e1*f2*g3*e4*f5*e6+e1*f2*g3*e4*e5*f6+e1*e2*g3*f4*f5*e6+e1*e2*g3*f4*e5*f6+e1*e2*g3*e4*f5*f6
\end{verb}

\begin{verb}
+f1*f2*e3*g4*e5*e6+f1*e2*f3*g4*e5*e6+f1*e2*e3*g4*f5*e6+f1*e2*e3*g4*e5*f6+e1*f2*f3*g4*e5*e6
\end{verb}

\begin{verb}
+e1*f2*e3*g4*f5*e6+e1*f2*e3*g4*e5*f6+e1*e2*f3*g4*f5*e6+e1*e2*f3*g4*e5*f6+e1*e2*e3*g4*f5*f6
\end{verb}

\begin{verb}
+f1*f2*e3*e4*g5*e6+f1*e2*f3*e4*g5*e6+f1*e2*e3*f4*g5*e6+f1*e2*e3*e4*g5*f6+e1*f2*f3*e4*g5*e6
\end{verb}

\begin{verb}
+e1*f2*e3*f4*g5*e6+e1*f2*e3*e4*g5*f6+e1*e2*f3*f4*g5*e6+e1*e2*f3*e4*g5*f6+e1*e2*e3*f4*g5*f6
\end{verb}

\begin{verb}
+f1*f2*e3*e4*e5*g6+f1*e2*f3*e4*e5*g6+f1*e2*e3*f4*e5*g6+f1*e2*e3*e4*f5*g6+e1*f2*f3*e4*e5*g6
\end{verb}

\begin{verb}
+e1*f2*e3*f4*e5*g6+e1*f2*e3*e4*f5*g6+e1*e2*f3*f4*e5*g6+e1*e2*f3*e4*f5*g6+e1*e2*e3*f4*f5*g6
\end{verb}

\noindent \hspace{1mm} \begin{verb}
Cfgh= [...]
\end{verb}


%
%
%
%
%
%
%
%
%
%

\noindent \hspace{1mm} \begin{verb}
(-1/6)*psie*Cgfe*psif*psih*Cfgh*psig
\end{verb}

}

The next result relates elements in $\Lambda_0(V^*)$ and $\Lambda_-(V^*)$ with {\em $\SU(3)$-structures} on $V$ (see \cite{Hitchin,Schulte}).
 
\begin{theorem} \label{thm:omega-psi}
Let $(\omega,\psi_-)\in \Lambda_0(V^*)\x \Lambda_-(V^*)$ such that 
\begin{equation}\label{eq:1}
\omega \wedge \psi_-=0.
\end{equation} 
Let $J=|\lambda(\psi_-)|^{-1/2} K_{\psi_-}$. If the tensor
$h(x,y)=\omega(x, J y)$ is positive definite,
then $(J,\omega)$ defines an $\SU(3)$-structure on $V$, and every $\SU(3)$-structure is obtained in this way. Taking $\psi_+=-J^*\psi_-$, we have that $\psi\coloneqq\psi_+ + i \psi_-$ is the complex volume form.
\end{theorem}

To apply Theorem \ref{thm:omega-psi}, we have to check that the quadratic form $h(x)=\omega(x,J(x))$
is positive definite. Equivalently, we look at $\hat h(x)=\omega (x, K(x))$, where $K=K_{\psi_-}$, $\psi=\psi_-$. This
is allowed since $J$ is a positive multiple of $K$.
Associated to $\omega=\sum c_{ij} e^{ij}$, we define the tensor
 $$
 \omega_{ef}= \sum c_{ij} (e^i\wedge f^j -e^j \wedge f^i) \in V_e^*\otimes V_f^*\, .
 $$

\begin{proposition}
 We have
 $$
 \hat h(x) \, \vol_{ef}= -  \psi_e \wedge C_{xfe}\wedge \psi_f \wedge \omega_{fx} \, ,
 $$
where $e,f$ are odd-degree variables, and $x$ is even degree.
\end{proposition}

\begin{proof}
 It is enough to prove this for $\omega=x^{ab}$, since the statement is linear in $\omega$. Let $x=\sum \alpha_i x_i \in V_x$,
and let $K(x)= \sum \kappa_{ij} \alpha_j x_i$. So 
\begin{equation}\label{eqn:55}
\hat h(x)=\omega(x,K(x))= \sum_\ell ( \alpha_a\alpha_{\ell} \kappa_{b\ell}-\alpha_b  \alpha_{\ell} \kappa_{a\ell} ).
 \end{equation}
By (\ref{eqn:formula}), we have $\psi_e \wedge C_{xfe}\wedge \psi_f= - \sum x^\ell \wedge k(f_\ell)\wedge \vol_e$.
 By formula (\ref{eqn:33}) below, we have
$k(f_\ell) \wedge f^a= - \kappa_{a \ell} \vol_f$, where $K(g_m)=\sum \kappa_{\ell m} g_\ell$. Therefore
 $$
 \psi_e \wedge C_{xfe}\wedge \psi_f \wedge  f^a = \sum x^\ell \kappa_{a\ell} \vol_{ef} \, .
 $$
Finally,
 $$
 \psi_e \wedge C_{xfe}\wedge \psi_f \wedge \frac12 (f^a\wedge x^b- f^b\wedge x^a) 
= \sum (x^{\ell b} \kappa_{a\ell} - x^{\ell a} \kappa_{b\ell})\vol_{ef}\, . 
 $$
Looking at (\ref{eqn:55}), we get the result. Recall that $x^{\ell b}=\frac12 (x^\ell \otimes x^b+x^b\otimes x^\ell)$ by
convention on symmetric tensors.
\end{proof}

We implement this in SageMath as follows. Note that for setting even degree, we set the degree equal to two (since it is not
allowed to have zero-degree variables).

{\tiny

\noindent \hspace{1mm} \begin{verb}
D.<e1,e2,e3,e4,e5,e6,f1,f2,f3,f4,f5,f6,x1,x2,x3,x4,x5,x6> = 
\end{verb}

\begin{verb} 
GradedCommutativeAlgebra(QQ,degrees=(1,1,1,1,1,1,1,1,1,1,1,1,2,2,2,2,2,2))
\end{verb}


\noindent \hspace{1mm} \begin{verb}
N=D.cdg_algebra({})
\end{verb}

\noindent \hspace{1mm} \begin{verb}
N.inject_variables()
\end{verb}

\noindent \hspace{1mm} \begin{verb}
psie=-e2*e4*e6-e1*e3*e6+e1*e4*e5-e2*e3*e5
\end{verb}

\noindent \hspace{1mm} \begin{verb}
psif=-f2*f4*f6-f1*f3*f6+f1*f4*f5-f2*f3*f5
\end{verb}

\noindent \hspace{1mm} \begin{verb}
omegafx=f1*x2-f3*x4+f5*x6-f2*x1+f4*x3-f6*x5
\end{verb}

\noindent \hspace{1mm} \begin{verb}
Cxfe= [...]
\end{verb}

\noindent \hspace{1mm} \begin{verb}
1/2*psie*Cxfe*psif*omegafx
\end{verb}
}

We will also need to compute $\psi_+$ explicitly out of $\omega,\psi_-$. We do this as follows.
Suppose that $V=\la x_1,\ldots, x_6\ra$ is a $6$-dimensional vector space.
The complex structure is $J=c K$, where $c=|\lambda(\psi_-)|^{-1/2}$, and $K=K_{\psi_-}$. So
 $\psi_+=-J^*\psi_- = - c^3 K^*\psi_-$.

By Proposition \ref{prop:Cefgh}, $K(g_m)= \sum \kappa_{\ell m} g_\ell$, where
 \begin{equation}\label{eqn:33}
\kappa_{\ell m} \vol_g=g^\ell \wedge k_{\psi_-}(g_m).
\end{equation}
For the action on forms, which is the dual one, we have
$K^*(g^\ell)=\sum \kappa_{\ell m} g^m$,
so that $K^*(g^\ell)\vol_g = \sum g^m \big( g^\ell \wedge k_{\psi_-}(g_m) \big)$. This means that for a $1$-form $\alpha$, and writing the map in different variables for source and target, $K^*:V_x\to V_h$, it is
$$
K^*(\alpha)\vol_x = \sum h^m \big( \alpha \wedge k_{\psi_-}(x_m) \big).
$$

For a $3$-form $\tau=\sum a_{ijk} x^{ijk}$, we write
 $$
 \tau_{xyz}=\sum a_{ijk} x^i\wedge y^j\wedge z^k \, .
 $$
Then  
 $$
 (K^*\tau) \, \vol_{xyz} = \sum h^{abc} \wedge k_{\psi_-}(x_a) \wedge k_{\psi_-}(y_b) \wedge k_{\psi_-}(z_c) \wedge \tau_{xyz}\, .
 $$
Using this for $\tau=\psi_-$ and formula (\ref{eqn:formula}), we have
 \begin{align*}
 \psi_+ &=-c^3(K^*\psi_-) \,  \vol_{xyzefg} =\\
&= -c^3 \psi_{e} \wedge C_{hxe} \wedge \psi_x \wedge
 \psi_{f} \wedge C_{hyf} \wedge \psi_y  \wedge
 \psi_{g} \wedge C_{hzg} \wedge \psi_z\wedge  \psi_{xyz}
 \end{align*}
where we have abreviated $\psi=\psi_-$.
The normalization can be obtained by means of
 $$
 \psi_-\wedge\psi_+ = \frac23 \omega^3\, .
$$

We implement this computation in SageMath as follows:

\medskip

{\tiny

\noindent \hspace{1mm} \begin{verb}
D.<e1, [...] ,h6,x1,x2,x3,x4,x5,x6,y1,y2,y3,y4,y5,y6,z1,z2,z3,z4,z5,z6>=GradedCommutativeAlgebra(QQ)
\end{verb}

\noindent \hspace{1mm} \begin{verb}
N=D.cdg_algebra({})
\end{verb}

\noindent \hspace{1mm} \begin{verb}
N.inject_variables()
\end{verb}

\noindent \hspace{1mm} \begin{verb}
psie=-e2*e4*e6+e1*e3*e6+e1*e4*e5+e2*e3*e5
\end{verb}

\noindent \hspace{1mm} \begin{verb}
[...]
\end{verb}

\noindent \hspace{1mm} \begin{verb}
psiz=-z2*z4*z6+z1*z3*z6+z1*z4*z5+z2*z3*z5
\end{verb}

\noindent \hspace{1mm} \begin{verb}
psixyz=-x2*y4*z6+x1*y3*z6+x1*y4*z5+x2*y3*z5
\end{verb}

\noindent \hspace{1mm} \begin{verb}
Chxe= [...]
\end{verb}

\noindent \hspace{1mm} \begin{verb}
Chyf= [...]
\end{verb}

\noindent \hspace{1mm} \begin{verb}
Chzg= [...]
\end{verb}

\noindent \hspace{1mm} \begin{verb}
psie*Chxe*psix*psif*Chyf*psiy*psig*Chzg*psiz*psixyz
\end{verb}
}

%
%
%
%
%
%
%

\section{Left-invariant $\Gtwo$-structures and $\SU(3)$-structures on Lie groups}

Let $G$ be a $7$-dimensional simply connected Lie group with Lie algebra $\lie{g}$. Then, a $\Gtwo$-structure 
on $G$ is \emph{left-invariant} if the corresponding
$3$-form is left-invariant. According to the discussion of Section \ref{tools}, a left-invariant $\Gtwo$-structure on $G$ is defined by a positive $3$-form $\varphi\in \Lambda^3({\lie{g}}^*)$, which can be written, in some orthonormal coframe $\{e^1,\dotsc, e^7\}$ of ${\mathfrak{g}}^*$, as 
\begin{equation}\label{eqn:3-forma G2}
 \varphi=e^{127}+e^{347}+e^{567}+e^{135}-e^{146} -e^{236}-e^{245}\,.
\end{equation}

We call this a $\Gtwo$-structure on $\lie{g}$. In this coframe, then,
\begin{equation}\label{eqn:4-forma G2}
*_\varphi\varphi=e^{1234}+e^{1256}+e^{1367}+e^{1457}+e^{2357}-e^{2467}+e^{3456}.
\end{equation}

A $\Gtwo$-structure on $\lie{g}$ is {\em coclosed} if $\varphi$ is coclosed, that is, if
\[
d*_\varphi\varphi=0\,,
\]
where $d$ denotes the Chevalley-Eilenberg differential on ${\lie{g}}^*$. A $\Gtwo$-structure on $\lie{g}$ is said to be {\em purely coclosed} if
\[
d*_\varphi\varphi=0 \quad \textrm{and} \quad \varphi\wedge d\varphi=0\,.
\]

We explain next how to construct a $\Gtwo$-structure on a $7$-dimensional Lie algebra, starting with a certain type of $\SU(3)$-structure on a codimension $1$ subspace and some extra data.

Let $\lie{g}$ be a $7$-dimensional Lie algebra with non-trivial center $\lie{z}(\fg)$. Let $V\subset\fg$ be a codimension 1 subspace, cooriented by $X\in\lie{z}(\fg)$; thus $X$ is a central vector with non-zero projection to $\fg/V$. Let $\omega\in\Lambda^2\fg^*$ and $\psi_-\in\Lambda^3\fg^*$ be such that
\begin{itemize}
\item $\imath_X\omega=0$;
\item $\imath_X\psi_-=0$;
\item they define an $\SU(3)$-structure on $V$.
\end{itemize}

Hence, denoting by $\bar{\omega}$ and $\bar{\psi}_-$ the pull-back to $V$ of the above tensors, we have $\bar{\omega}\in \Lambda_0(V^*)$, $\bar{\psi}_-\in\Lambda_-(V^*)$ and $\bar{\omega}\wedge\bar{\psi}_-=0$. This determines $\bar{\psi}_+\in\Lambda^3(V^*)$. Extend $\bar{\psi}_+$ to an element $\psi_+\in\Lambda^3\fg^*$ by declaring $\imath_X\psi_+=0$. Finally, let $\eta\in\fg^*$ be such that
$\eta(X)\neq 0$.

It follows that $\varphi=\omega\wedge\eta+\psi_+$ is a $\Gtwo$-form on $\fg$; moreover, if $h$ denotes the induced $\SU(3)$-metric on $V$, then the $\Gtwo$-metric on $\fg$ is $g=g_\varphi=h+\eta\ox\eta$. Clearly, $*_\varphi\varphi=\frac{\omega^2}{2}+\psi_-\wedge\eta$. We want to find sufficient conditions on $(\omega,\psi_-,\eta)$ in order for the $\Gtwo$-structure to be (purely) coclosed.

\begin{theorem}\label{thm:construction}
In the above setting, the $\Gtwo$-structure is coclosed if
\begin{enumerate}
\item $d\psi_-=0$;
\item $\omega\wedge d\omega=\psi_-\wedge d\eta$.
\end{enumerate}
Furthermore, the coclosed $\Gtwo$-structure is pure if
\begin{enumerate}
\setcounter{enumi}{2}
\item $\omega^2 \wedge d\eta =-2 \psi_+ \wedge d\omega$.
\end{enumerate}
\end{theorem}

\begin{proof}
We compute 
\[
d*_\varphi\varphi = \omega\wedge d\omega+d\psi_-\wedge\eta-\psi_-\wedge d\eta=(\omega\wedge d\omega-\psi_-\wedge d\eta)+d\psi_-\wedge\eta=0\,.
\]

The $\Gtwo$-form is $\varphi=\omega \wedge \eta + \psi_+$, so that $d\varphi=d\omega \wedge \eta +\omega \wedge d\eta +d\psi_+$. Hence, 
\begin{align*}
\varphi\wedge d\varphi & = (\omega \wedge \eta + \psi_+)\wedge (d\omega \wedge \eta +\omega \wedge d\eta +d\psi_+)\\
&=(\omega^2\wedge d\eta+\psi_+\wedge d\omega+d\psi_+\wedge\omega)\wedge\eta+\psi_+\wedge \omega \wedge d\eta+\psi_+\wedge d\psi_+\,.
\end{align*}
Now $\psi_+\wedge\omega=0$ and hence $d\psi_+\wedge \omega=\psi_+\wedge d\omega$. Also since $X$ is central, $\imath_Xd\psi_+=0$, hence $\imath_X(\psi_+\wedge d\psi_+)=0$ and so $\psi_+\wedge d\psi_+=0$ since it is a $7$-form. All in all, this implies that 
\[
\varphi\wedge d\varphi =(\omega^2\wedge d\eta+2\psi_+ \wedge d\omega) \wedge \eta=0\,,
\]
and the result follows.
\end{proof}

The next example shows how to apply Theorem \ref{thm:construction} in practice.

\begin{example}
Let us consider the Lie algebra $\fg=37B=(0,0,0,0,12,23,34)$ in the notation of \cite{Gong}; this means that $\fg$ is $7$-dimensional and that it admits a basis $\{e_1,\ldots,e_7\}$ such that, in terms of the dual basis $\{e^1,\ldots,e^7\}$, the Lie algebra structure is given by $de^5=e^{12}$, $de^6=e^{23}$ and $de^7=e^{34}$. Consider the subspace $V=\mathrm{span}(e_1,e_2,e_3,e_4,e_6,e_7)\subset\fg$ and the central vector $X=e_5$. Set
\begin{itemize}
\item $\omega=e^{13}+e^{24}-e^{67}$;
\item $\psi_-=e^{127}-e^{146}+e^{236}-e^{347}$;
\item $\eta=e^5+e^7$.
\end{itemize}
Then $(\omega,\psi_-)$ defines an $\SU(3)$-structure on $V$ with induced metric $h=\sum_{i\neq 5}e^i\ox e^i$, and $\psi_+=e^{126}+e^{147}-e^{346}-e^{237}$. Hence $\varphi=\omega\wedge\eta+\psi_+$ defines a $\Gtwo$-structure on $\fg$ with $\Gtwo$-metric $g=g_\varphi=\sum_{i=1}^6e^i\ox e^i+2e^7\ox e^7+e^5\ox e^7+e^7\ox e^5$. Since conditions (1), (2), and (3) of Theorem \ref{thm:construction} are satisfied, the $\Gtwo$-structure is purely coclosed: $d*_\varphi\varphi=0$ and $\varphi\wedge d\varphi=0$. 
\end{example}

\subsection{Nilpotent Lie algebras and nilmanifolds}
A nilmanifold is a compact manifold of the form $\Gamma\backslash G$, where $G$ is a connected, simply connected, nilpotent Lie group and $\Gamma$ is a cocompact discrete subgroup (a lattice). 
By a result of Mal'cev \cite{Malcev} we know that if $\lie{g}$ is nilpotent with rational structure constants, then the associated connected, simply connected nilpotent Lie group $G$ admits a lattice $\Gamma$. Therefore, a left-invariant $\Gtwo$-structure on $G$ determines a $\Gtwo$-structure on the nilmanifold $\G\backslash G$. Moreover, if the left-invariant $\Gtwo$-structure on $G$ is coclosed (resp.~purely coclosed), the same holds for the induced $\Gtwo$-structure on $\G\backslash G$.
The left-invariant forms on $G$ are given by the corresponding forms on $\Lambda^*(\lie{g}^*)$, so the computations can be done on the Lie algebra. Notice that a nilpotent Lie algebra has non-trivial center.

We use Gong's classification \cite{Gong} of $7$-dimensional nilpotent Lie algebras that appears in the Appendix. We translate the Lie algebra brackets
into the DGA structure for $\Lambda(\lie{g}^*)$ for each of them. 
Gong's classification \cite{Gong} of $7$-dimensional nilpotent Lie algebras is done over the real numbers. For those
nilpotent Lie algebras that have a single representative, the structure constants are actually rational numbers, so they define nilmanifolds. For those nilpotent Lie algebras that appear in families depending on a real parameter $\lambda$, it is not clear that $\lambda$ must be rational in order to define a nilmanifold, since a different choice of generators may have rational structure constants. We do not care about this issue, since we produce $\Gtwo$-structures for all values of the parameter $\lambda$. In the cases that there are special values to be treated separately, these are always rational.

Notice that the forms $\omega,\psi_-,\eta$ can be defined with real coefficients, and that it is only the structure constants of the DGA that should
be rational to have a nilmanifold with a (purely coclosed) $\Gtwo$-structure.

\section{Obstructions for coclosed $\Gtwo$-structures on NLA}\label{sec:coclosed}

We use a converse of Theorem \ref{thm:construction}, which follows from 
\cite[Chapter 1, Proposition 4.5]{Schulte} and \cite[Proposition 3.1]{BFF}. 
First of all, suppose $\varphi$ is a $\Gtwo$-structure on a $7$-dimensional Lie algebra and let $g_\varphi$ be the $\Gtwo$-metric. Pick a vector $X$ of length $1$. It follows from that $X^\perp$ has an $\SU(3)$-structure $(\omega, \psi_-)$ given by 
$\omega\coloneqq \imath_X\varphi$ and $\psi_-\coloneqq-\imath_X\phi$, where $\phi=\ast_\varphi\varphi$.
Let $\eta$ be the metric dual of $X$. Then $\phi= \frac12 \omega^2 + \psi_-\wedge \eta$. We have the following

\begin{proposition}\label{prop1}
Let $\lie{g}$ be a $7$-dimensional Lie algebra with non-trivial center and let $X$ be a unit central vector. Let $\lie{h}=\lie{g}/\la X\ra$ be the quotient
$6$-dimensional Lie algebra. If $\varphi$ is a $3$-form defining a coclosed $\Gtwo$-structure on $\lie{g}$, then it
determines an $\SU(3)$-structure $(\omega,\psi_-)$ on $\mathfrak{h}$ such that 
%
%
 \begin{align*}
  d\psi_- &= 0, \\
  \omega \wedge d\omega &= \psi_-\wedge d\eta \, .  
  \end{align*} 
\end{proposition}

In this section we obtain obstructions to the existence of coclosed $\Gtwo$-structures. In the next subsections we shall provide three of them.

\subsection{First obstruction}

For the first obstruction we use the following result, which appears in \cite[Lemma 3.3]{BFF}.

\begin{lemma}\label{obs3}
Let $\lie{g}$ be a $7$-dimensional Lie algebra. If for every closed $4$-form $\kappa$ on $\lie{g}$ there are linearly independent
vectors $X$ and $Y$ in $\lie{g}$ such that $(\imath_{X}\imath_{Y}\kappa)^2=0$, 
then $\lie{g}$ does not admit coclosed $\Gtwo$-structures.
\end{lemma}

This is used in the following form

\begin{corollary} \label{cor:obs3}
Let $\lie{g}$ be a $7$-dimensional Lie algebra. Take the cohomology
$H^4(\Lambda (\lie{g}^*))=\la [z_\alpha]\ra$. Suppose that there exist 
linearly independent vectors $X,Y \in \lie{g}$, with $Y\in\lie{z}(\lie{g})$, such that $\imath_{X}\imath_{Y} z_\alpha\in U$ for a subspace $U\subset\Lambda^2\lie{g}^*$ such that $\Lambda^2 U=0$.
Then $\lie{g}$ does not admit any coclosed $\Gtwo$-structure.
\end{corollary}

\begin{proof}
Thanks to Lemma \ref{obs3}, it suffices to check that for any $\kappa\in\Lambda^4\lie{g}^*$ closed we have $\imath_{X}\imath_{Y}\kappa\in U$, and hence $(\imath_{X}\imath_{Y}\kappa)^2=0$. Any such $\kappa$ is a linear combination of the $z_\alpha$'s and an exact $4$-form. For the former, the result holds by assumption. If $\kappa$ is an exact $4$-form, then $\imath_Y\kappa=0$, because $Y\in\lie{z}(\lie{g})$.
\end{proof}

\subsection{Second obstruction}
The second obstruction uses the following result, which appears in \cite[Lemma 3.4]{BFF}.

\begin{lemma}\label{lem2}
Let $(h,J)$ be an 
almost Hermitian structure on a $6$-dimensional oriented vector space $V$, 
with orthogonal complex structure $J$, Hermitian metric $h$ and fundamental $2$-form $\omega$. Then, for any $J$-invariant $4$-dimensional subspace $W \subset V$, we have that 
$(*\omega)|_W\neq 0$.
\end{lemma}

We use it in the following form

\begin{corollary}\label{cor:lem2}
Let $\lie{g}$ be a $7$-dimensional nilpotent Lie algebra and let $\{e_1,\ldots,e_7\}$ be a nilpotent basis, with dual basis $\{e^1,\ldots,e^7\}$. Take a list of generators of the space of closed $4$-forms $z_\alpha \in \Lambda^4(\lie{g}^*)$. Suppose that $z_\alpha \in \la e^1,e^2\ra \wedge \Lambda^3(\lie{g}^*)$ for all $\alpha$.
Then $\lie{g}$ does not admit coclosed $\Gtwo$-structures.
\end{corollary}

\begin{proof}
 Suppose that $\varphi$ is a coclosed $\Gtwo$-structure, and let $\phi=*_\varphi\varphi$, so that $d\phi=0$. Then $\phi\in\Lambda^4(\lie{g}^*)$ is
a closed $4$-form.
We fix $X=e_7$, and consider the quotient algebra ${\mathfrak h}={\mathfrak g}/\la X\ra$. By Proposition \ref{prop1}, the forms
$\omega=\imath_X\varphi$ and $\psi_-=-\imath_X\phi$ determine an $\SU(3)$-structure on ${\mathfrak h}=\la e_1,e_2,e_3,e_4,e_5,e_6\ra$.
Then $\psi_- = -\imath_X \phi \in \la e^1,e^2\ra \wedge \Lambda^2(\lie{h}^*)$.

Let us see that $W=\la e_3,e_4,e_5,e_6\ra$ is a complex subspace of $\lie{h}$.
 It is enough to see that $K=K_{\psi_-}$ leaves $W$ invariant, since $J$ is a multiple of it.
Recall that $k(x)=\imath_x \psi_- \wedge \psi_-$ and $K=\mu\circ k$. To check that $W$ is invariant,
we need to check that $k(e_j) \wedge e^1 =0$, $k(e_j) \wedge e^2 =0$, for $j=3,\ldots,6$. 
This is clear since both $\imath_{e_j}\psi_-$ and $\psi_-$ contain at least one $e^1,e^2$, and hence a product with $e^1$ or $e^2$ kills it.

By Lemma \ref{obs3}, it must be $\omega^2|_W \neq 0$. This means that $\omega^2(e_3,e_4,e_5,e_6)\neq 0$; in other words, $\omega^2$ contains the monomial $e^{3456}$. By Proposition \ref{prop1}, we have
$ \omega^2 = 2(\phi- \psi_-\wedge \eta)$,
where $\eta$ is a $1$-form with a component $e^7$. By assumption, $\phi$ cannot contain
$e^{3456}$ (it always contains either $e^1$ or $e^2$), and $\psi_-$ also contains $e^1$ or $e^2$.
Thus $e^{3456}$ cannot appear in $\omega^2$. This is a contradiction.
\end{proof}

\subsection{Third obstruction}
This method is based on the last equation of Proposition \ref{prop1}. We fix $X=e_7$. We compute a basis $\{z_\alpha\}$ of the closed $3$-forms in $\lie{h}^*=\la e^1,\ldots,e^6\ra$. Therefore for a closed $3$-form, we can write $\tau= \sum a_\alpha z_\alpha$.

\begin{proposition}\label{prop:third}
Suppose we have elements $w_1,\ldots, w_\ell\in \Lambda^5(\lie{h}^*)$, and let $W$ be a subspace such that $\Lambda^5(\lie{h}^*)=W\oplus \la w_1,\ldots, w_\ell\ra$. Suppose furthermore that for any closed $2$-form $\beta$ and 
closed $3$-form $\tau$ on $\lie{h}^*$, we have 
 $$
 \beta\wedge d\beta ,\tau \wedge de^j \in W,  \ j=1,\ldots, 6.
 $$
Then define the linear subspace
\[
H=\left\{ (a_\alpha) \, | \, \sum a_\alpha z_\alpha \wedge de^7 \in W\right\}.
\]
If $\lambda ( \sum a_\alpha z_\alpha)\geq 0$ for all $(a_\alpha)\in H$, then there is no coclosed $\Gtwo$-structure on $\lie{g}$.
\end{proposition}

\begin{proof}
 If $\varphi$ is a coclosed $\Gtwo$-structure on $\fg$, then $\psi_-$ is closed and thus $\psi_-= \sum a_\alpha z_\alpha$.
Next $\eta$ must have the 
form $\sum_{i=1}^7 b_i e^i$, with $b_7\neq 0$. Hence
\[
\psi_+\wedge de^7 =b_7^{-1}\Big( \omega \wedge d\omega -\sum_{i=1}^6 b_i \psi_-\wedge de^i \Big) \in W,
\]
by assumption. Hence $(a_\alpha) \in H$, and so $\lambda(\psi_-)\geq 0$, which is a contradiction.
\end{proof}
%
%
%
%
%
%
%
%
\section{7D nilpotent Lie algebras with purely coclosed $\Gtwo$-structures}

\subsection{Decomposable nilpotent Lie algebras} \label{subsec:dec} 7-dimensional nilpotent Lie algebras are listed in Table \ref{table:7Ddec}. In \cite{BFF} the authors identify $7$-dimensional decomposable nilpotent Lie algebras that admit a coclosed $\Gtwo$-structure. The following result summarizes Proposition 4.1, Proposition 4.2, and Theorem 4.3 in \cite{BFF}.

\begin{theorem}\label{thm:decomposable-cocalibrated}
Among the 35 decomposable nilpotent Lie algebras, of dimension $7$, those that have a coclosed $\Gtwo$-structure are $\fn_i$, $i=1,\ldots,24$, and those who do not admit any coclosed $\Gtwo$-structure are $\fg_i$, $i=1,\ldots,8$ and $\mathfrak{l}_i$, $i=1,2,3$.
\end{theorem}

We refer to Table \ref{table:7Ddec} for the structure equations of these Lie algebras. The 24 decomposable nilpotent Lie algebras which admit a coclosed $\Gtwo$-structure have the form $\lie{n}=\lie{h}\oplus\RR$, where $\lie{h}$ is a 6-dimensional Lie algebra, generated by $\{e_1,\ldots,e_6\}$, endowed with a {\em half-flat} $\SU(3)$-structure $(\omega,\psi_-)$, that is, $d(\omega^2)=0$ and $d\psi_-=0$. The $\Gtwo$-metric is $g=h+e^7\otimes e^7$, where $h$ is the $\SU(3)$-metric on $\lie{h}$ and $e_7$ is a generator of the factor $\RR$. 
 Then $*_\varphi\varphi=\frac{1}{2}\omega^2+\psi_-\wedge e^7$ is automatically coclosed; indeed, conditions (1) and (2) of Theorem \ref{thm:construction} are satisfied with $\eta=e^7$, which is closed. Condition (3), however, need not hold. Nevertheless, we prove the following result:

\begin{theorem}\label{thm:decomposable-pure}
Every $7$-dimensional decomposable nilpotent Lie algebra admitting a coclosed $\Gtwo$-structure also admits a purely coclosed one, except for $\fn_2$.
\end{theorem}

\begin{proof}
We apply Theorem \ref{thm:construction} to each of the Lie algebras $\fn_i$, $i=1,3,\ldots,24$, and exhibit explicitly the required tensors $\omega$, $\psi_-$ and $\eta$ satisfying conditions (1), (2) and (3).  The results are contained in Tables \ref{table:0a} and \ref{table:0b}. That $\fn_2=\mathfrak{h}_3\oplus\RR^4$ admits no purely coclosed $\Gtwo$-structure follows from \cite[Corollary 4.3]{BMR}.
\end{proof}

\begin{remark}
As we shall see, $\fn_2$ is the only 7-dimensional nilpotent Lie algebra of nilpotency step $\leq 4$ which admits a coclosed $\Gtwo$-structure but no purely coclosed structures. It turns out that $\fn_2$ admits the following purely coclosed $\Gtwo^*$-structure:
\[
\ast_\varphi\varphi=-e^{1234}+e^{2356}+e^{1456}+e^{1357}+e^{2457}-e^{1367}+e^{2467}\,.
\]
\end{remark}

\begin{table}[H]
\centering
\caption{Purely coclosed $\Gtwo$-structures on decomposable NLAs $-$ 1}\label{table:0a}
\vspace{0.25 cm}
{\tabulinesep=1.2mm
\begin{tabu}{l|c|c|c}
\toprule[1.5pt]
NLA & $\omega$ & $\psi_-$ & $\eta$\\
\specialrule{1pt}{0pt}{0pt}
$\fn_1$ & $e^{12}-e^{34}+e^{56}$ & $e^{136}+e^{145}+e^{235}-e^{246}$ & $e^7$\\
\specialrule{1pt}{0pt}{0pt}
$\fn_3$ & $-e^{12}-e^{34}-e^{45}+e^{57}$ & $e^{135}+e^{147}+e^{237}-e^{257}-e^{245}$ & $-e^6$\\
\specialrule{1pt}{0pt}{0pt}
$\fn_4$ & $e^{12}-e^{14}-2e^{34}-e^{56}$ & $e^{135}-e^{146}+e^{236}+e^{245}+e^{346}$ & $e^7+2e^5$\\
\specialrule{1pt}{0pt}{0pt}
$\fn_5$ & $e^{13}+e^{24}-e^{56}$ & $e^{125}-e^{345}+e^{146}-e^{236}+e^{126}$ & $e^7$\\
\specialrule{1pt}{0pt}{0pt}
$\fn_6$ & $e^{12}+e^{34}+e^{56}$ & $e^{135}-e^{146}-e^{236}-e^{245}$ & $e^7$\\
\specialrule{1pt}{0pt}{0pt}
$\fn_7$ & $e^{13}+e^{24}-e^{56}$ & $-e^{125}-2e^{146}+e^{236}+e^{345}$ & $e^7$\\
\specialrule{1pt}{0pt}{0pt}
$\fn_8$ & $\begin{array}{c}
-e^{12}+e^{15}-e^{26}-2e^{34}\\[1.5pt]
+3e^{36}-e^{45}
\end{array}$ & $\begin{array}{c}
e^{135}+e^{146}+e^{234}\\
-e^{256}+e^{236}-e^{345}
\end{array}$ & $e^7-3e^6$\\
\specialrule{1pt}{0pt}{0pt}
$\fn_9$ & $\begin{array}{c}
-e^{14}+e^{15}+14e^{26}+e^{34}\\[1.5pt]
-7e^{35}-e^{45}
\end{array}$ & $e^{123}+e^{125}-e^{156}+e^{245}+e^{346}$ & $e^7+28e^6$\\
\specialrule{1pt}{0pt}{0pt}
$\fn_{10}$ & $\begin{array}{c}
e^{17}-2e^{23}-e^{45}+\frac{4}{5}(e^{27}-e^{13})\\[1.5pt]
+\frac{3}{5}(e^{47}-e^{15})
\end{array}$ & $e^{125}+e^{134}-e^{247}+e^{357}$ & $e^6+\frac{4}{5}e^5$\\
\specialrule{1pt}{0pt}{0pt}
$\fn_{11}$ & $e^{16}+e^{24}-e^{35}-e^{45}$ & $
-e^{123}+e^{135}-e^{145}-e^{256}-e^{346}$ & $e^7+e^6$\\
\specialrule{1pt}{0pt}{0pt}
$\fn_{12}$ & $\begin{array}{c}
5e^{12}+10e^{14}-3e^{16}-3e^{23}-4e^{25}\\[1.5pt]
-e^{34}+e^{36}-e^{45}-e^{56}
\end{array}$ & $\begin{array}{c}
e^{123}+e^{134}+e^{136}\\
+e^{156}+e^{246}+e^{345}
\end{array}$ & $e^7$\\
\specialrule{1pt}{0pt}{0pt}
$\fn_{13}$ & $\begin{array}{c}
\frac{1}{2}e^{12}+e^{13}+\frac{1}{2}e^{16}+\frac{1}{2}e^{23}\\[1.5pt]
+e^{26}+\frac{1}{2}e^{36}-2e^{45}
\end{array}$ & $-e^{124}+e^{156}+e^{235}-e^{346}$ & $\begin{array}{c}e^7+e^5\\+e^4\end{array}$\\
\specialrule{1pt}{0pt}{0pt}
$\fn_{14}$ & $e^{13}+e^{26}-e^{45}$ & $-e^{124}+e^{156}+e^{235}-e^{346}$ & $e^7$\\
\bottomrule[1pt]
\end{tabu}}
\end{table} 

\begin{table}[H]
\centering
\caption{Purely coclosed $\Gtwo$-structures on decomposable NLAs $-$ 2}\label{table:0b}
\vspace{0.25 cm}
{\tabulinesep=1.2mm
\begin{tabu}{l|c|c|c}
\toprule[1.5pt]
NLA & $\omega$ & $\psi_-$ & $\eta$\\
\specialrule{1pt}{0pt}{0pt}
$\fn_{15}$ & $\begin{array}{c}
e^{12}+e^{15}-\frac{1}{2}e^{16}\\[1.5pt]
+\frac{1}{2}e^{25}-e^{26}+e^{34}
\end{array}$ & $-e^{123}+e^{146}+e^{236}+e^{245}-e^{356}$ & $e^7$\\
\specialrule{1pt}{0pt}{0pt}
$\fn_{16}$ & $\begin{array}{c}
3e^{13}+e^{14}+e^{16}-e^{25}\\
+e^{34}-e^{36}-e^{46}
\end{array}$ & $-e^{124}+e^{156}+e^{236}+e^{345}$ & $e^7-e^5$\\
\specialrule{1pt}{0pt}{0pt}
$\fn_{17}$ & $e^{13}+e^{14}+e^{25}-e^{36}$ & $-e^{123}-e^{156}+e^{236}+e^{246}+e^{345}$ & $e^7+e^5$\\
\specialrule{1pt}{0pt}{0pt}
$\fn_{18}$ & $\begin{array}{c}
e^{13}+e^{24}+e^{26}-\frac{1}{2}e^{34}\\
-e^{36}+\frac{1}{2}e^{45}-\frac{1}{2}e^{46}-e^{56}
\end{array}$ & $\begin{array}{c}
e^{123}+e^{125}-e^{126}+e^{146}\\
-e^{235}-e^{236}-e^{345}
\end{array}$ & $e^7+\frac{1}{2}e^4$\\
\specialrule{1pt}{0pt}{0pt}
$\fn_{19}$ & $\begin{array}{c}
\frac{5}{6}e^{13}+e^{25}+e^{46}\\[1.5pt]
+\frac{1}{2}(e^{15}+2e^{23}+e^{24}+e^{56})
\end{array}$ & $e^{124}-e^{156}+2e^{236}+e^{345}$ & $\begin{array}{c}e^7+\frac{1}{4}e^5\\[1.5pt]+\frac{3}{2}e^4\end{array}$\\
\specialrule{1pt}{0pt}{0pt}
$\fn_{20}$ & $e^{13}+3e^{24}-e^{56}+\frac{3}{2}(e^{26}+e^{45})$ & $-e^{125}-e^{146}+e^{236}+e^{345}$ & $e^7$\\
\specialrule{1pt}{0pt}{0pt}
$\fn_{21}$ & $-e^{14}+\frac{1}{2}(e^{15}+e^{24})+e^{25}+e^{36}$ & $-e^{123}+e^{156}+e^{246}-e^{256}-e^{345}$ & $e^7$\\
\specialrule{1pt}{0pt}{0pt}
$\fn_{22}$ & $e^{15}+e^{24}+e^{36}$ & $\begin{array}{c}
e^{123}-e^{134}+e^{146}\\
-e^{235}-e^{256}-2e^{345}
\end{array}$ & $e^7$\\
\specialrule{1pt}{0pt}{0pt}
$\fn_{23}$ & $e^{15}+e^{24}+e^{36}$ & $e^{123}-e^{146}+e^{256}+e^{345}$ & $e^7$\\
\specialrule{1pt}{0pt}{0pt}
$\fn_{24}$ & $-e^{14}+e^{25}+e^{36}-e^{56}$ & $2e^{123}-2e^{135}-e^{156}-e^{246}+2e^{345}$ & $e^7$\\
\bottomrule[1pt]
\end{tabu}}
\end{table} 

The following routine computes $\omega\wedge \psi_-$ and $\psi_-\wedge\psi_+-\frac{2}{3}\omega^3$ to ensure that $(\omega,\psi_-)$ defines a normalized $\SU(3)$-structure via Theorem \ref{thm:omega-psi}. It also computes $d\psi_-$, $\omega\wedge d\omega - \psi_-\wedge d\eta$
and $\omega^2\wedge d\eta+2\psi_+\wedge d\omega$, to check that the $\Gtwo$-structure $\varphi=\omega \wedge\eta+\psi_+$ is purely coclosed, via Theorem \ref{thm:construction}.

To check that the forms in Tables \ref{table:0a} and \ref{table:0b} satisfy the required conditions, we use a SageMath worksheet \cite{Sage}. The
worksheets are available in \cite{worksheets}. We include the worksheet for the Lie algebra $\fn_{15}$.

\medskip

{\tiny

\noindent \hspace{1mm} \begin{verb}
A.<x1,x2,x3,x4,x5,x6,x7> = GradedCommutativeAlgebra(QQ)
\end{verb}

\noindent \hspace{1mm} \begin{verb}
M=A.cdg_algebra({ x4:x1*x2, x5:x1*x3, x6: x1*x4+x3*x5 })
\end{verb}

\noindent \hspace{1mm} \begin{verb}
M.inject_variables()
\end{verb}

\noindent \hspace{1mm} \begin{verb}
omega=x1*x7+x2*x3+2*x3*x4-x1*x4+x3*x5+x2*x7-4*x1*x5
\end{verb}

\noindent \hspace{1mm} \begin{verb}
psi=x1*x2*x4-x1*x2*x5+x1*x4*x7+x1*x5*x7+x2*x3*x4+x2*x3*x5-x3*x5*x7
\end{verb}

\noindent \hspace{1mm} \begin{verb}
psiplus=-2*(3*x1*x2*x4-x2*x3*x4+5*x1*x2*x5+3*x2*x3*x5-x1*x4*x7+2*x3*x4*x7+3*x1*x5*x7+x3*x5*x7)
\end{verb}

\noindent \hspace{1mm} \begin{verb}
eta=x6
\end{verb}

\noindent \hspace{1mm} \begin{verb}
omega*psi
\end{verb}

\noindent \hspace{1mm} \begin{verb}
psi*psiplus-(2/3)*omega^3
\end{verb}

\noindent \hspace{1mm} \begin{verb}
psi.differential()
\end{verb}

\noindent \hspace{1mm} \begin{verb}
omega*omega.differential()-psi*eta.differential()
\end{verb}

\noindent \hspace{1mm} \begin{verb}
omega^2*eta.differential()+2*psiplus*omega.differential()
\end{verb}

}



\subsection{Indecomposable 2-step nilpotent Lie algebras} \label{subsec:sage}

The case of indecomposable $2$-step nilpotent Lie algebras was also tackled in \cite{BFF}. According to Gong's classification \cite{Gong}, there are 9 indecomposable $2$-step nilpotent Lie algebras, see Table \ref{7d-ind-2step}.

We have the following result:
\begin{theorem}\cite[Theorem 5.1]{BFF}\label{indecomposable-cocalibrated}
Up to isomorphism, the indecomposable $2$-step nilpotent Lie algebras carrying a coclosed $\Gtwo$-structure are $17$, $37A$, $37B$, $37B_1$, $37C$, $37D$, and $37D_1$.
\end{theorem}

It follows from \cite[Theorem 1.4]{BMR} that these Lie algebras also admit a purely coclosed $\Gtwo$-structure. We are going to give a different proof of this fact, by applying Theorem \ref{thm:construction} and exhibiting explicitly the required tensors $\omega$, $\psi_-$ and $\eta$ satisfying conditions (1), (2) and (3).

\begin{table}[H]
\centering
\caption{Purely coclosed $\Gtwo$-structures on indecomposable 2-step NLAs}\label{table:1}
\vspace{0.25 cm}
{\tabulinesep=1.2mm
\begin{tabu}{l|c|c|c}
\toprule[1.5pt]
NLA & $\omega$ & $\psi_-$ & $\eta$\\
\specialrule{1pt}{0pt}{0pt}
$17$ & $-e^{12}+\frac{1}{2}e^{34}-e^{56}$ & $e^{135}+e^{146}-e^{236}+e^{245}$ & $e^7$\\
\specialrule{1pt}{0pt}{0pt}
$37A$ & $-e^{12}+e^{34}-e^{56}$ & $-e^{136}+e^{145}-e^{235}-e^{246}$ & $-e^7+e^5$\\
\specialrule{1pt}{0pt}{0pt}
$37B$ & $e^{13}+e^{24}-e^{67}$ & $e^{127}-e^{146}+e^{236}-e^{347}$ & $e^7+e^5$\\
\specialrule{1pt}{0pt}{0pt}
$37B_1$ & $-e^{12}-2e^{34}+e^{67}$ & $-e^{137}+e^{146}+e^{236}+e^{247}$ & $2e^5$\\
\specialrule{1pt}{0pt}{0pt}
$37C$ & $e^{12}-e^{34}-e^{67}$ & $e^{136}-e^{147}+e^{237}+e^{246}$ & $e^5$\\
\specialrule{1pt}{0pt}{0pt}
$37D$ & $-e^{12}+e^{34}-e^{67}$ & $e^{136}-e^{146}-e^{137}+2e^{147}+e^{246}-e^{237}$ & $e^5$\\
\specialrule{1pt}{0pt}{0pt}
$37D_1$ & $e^{12}+e^{34}+e^{67}$ & $e^{137}+e^{146}+e^{236}-e^{247}$ & $e^5$\\
\bottomrule[1pt]
\end{tabu}}
\end{table} 

Therefore we confirm:

\begin{theorem}\label{2step-indecomposable-pure}
All $7$-dimensional indecomposable 2-step nilpotent Lie algebra admit a purely coclosed $\Gtwo$-structure, except for $27A$ and $27B$.
\end{theorem}

\subsection{Indecomposable 3-step nilpotent Lie algebras}

We move now to the unknown territory of 7D indecomposable 3-step nilpotent Lie algebras. According to Gong \cite{Gong} there are 52 of them, listed in Tables \ref{7d-ind-3step} and \ref{7d-ind-3step-2}. Two of them, $147E(\lambda)$ and $147E_1(\lambda)$, depend on a real parameter $\lambda$; there are conditions on the value of the parameter, which are included in Appendix \ref{appendix}.

We have the following result:

\begin{theorem}\label{3step-indecomposable-pure}
All $7$-dimensional indecomposable 3-step nilpotent Lie algebra admit a purely coclosed $\Gtwo$-structure, except for: $247E$, $247E_1$, $247G$, $247H$, $247H_1$, $247K$, $247R$, $247R_1$, $257B$, $257D$, $257E$, $257G$, $257H$, $257K$, $257L$, $357B$ and $357C$.
\end{theorem}

{\footnotesize 
\begin{table}[H]
\centering
\caption{Purely coclosed $\Gtwo$-structures on indecomposable 3-step NLAs $-$ 1}\label{table:2}
\vspace{0.25 cm}
{\tabulinesep=1.2mm
\begin{tabu}{l|c|c|c}
\toprule[1.5pt]
NLA & $\omega$ & $\psi_-$ & $\eta$\\
\specialrule{1pt}{0pt}{0pt}
$137A$ & $e^{13}-e^{24}+e^{56}$ & $e^{126}-e^{145}-e^{235}-e^{345}+e^{346}$ & $e^7$\\
\specialrule{1pt}{0pt}{0pt}
$137A_1$ & $e^{12}+e^{34}+e^{56}$ & $
-e^{246}+e^{136}+e^{145}+e^{235}$ & $e^7$\\
\specialrule{1pt}{0pt}{0pt}
$137B$ & $\begin{array}{c}
e^{13}-\frac{3}{4}e^{24}+e^{56}\\[1.5pt]
-\frac{1}{2}(e^{15}+e^{36})
\end{array} $
& $e^{126}-e^{145}-e^{235}-e^{345}+e^{346}$ & $e^7$\\
\specialrule{1pt}{0pt}{0pt}
$137B_1$ & $e^{12}+e^{34}+e^{56}$ & $2e^{136}+e^{145}+e^{235}-e^{246}$ & $-\frac{1}{2}e^7$\\
\specialrule{1pt}{0pt}{0pt}
$137C$ & $\begin{array}{c}
e^{12}+\frac{2}{3}e^{13}-2e^{24}\\
-e^{34}+e^{56}
\end{array}$ & $e^{126}-e^{145}-e^{235}+e^{346}$ & $e^7-e^5$\\
\specialrule{1pt}{0pt}{0pt}
$137D$ & $e^{13}-e^{24}+e^{56}$ & $e^{126}-e^{145}-e^{235}+e^{346}$ & $-3e^7$\\
\specialrule{1pt}{0pt}{0pt}
$147A$ & $e^{13}-e^{25}-e^{46}$ & $\begin{array}{c}
-e^{123}-2e^{126}+e^{145}-e^{234}\\
+e^{246}-e^{356}
\end{array}$ & $-\frac{1}{2}(e^7+e^5)$\\
\specialrule{1pt}{0pt}{0pt}
$147A_1$ & $e^{13}-e^{25}-e^{46}$ & 
$\begin{array}{c}
-e^{123}-2e^{126}+e^{145}-e^{234}\\
+e^{246}-e^{356}
\end{array}$ & $e^7-2e^4$\\
\specialrule{1pt}{0pt}{0pt}
$147B$ & $\begin{array}{c}
-e^{13}+e^{14}+e^{16}+e^{23}\\
+3e^{24}+e^{25}+e^{46}+e^{56}
\end{array}$ & 
$e^{124}+e^{126}+e^{145}+e^{235}+e^{256}+e^{346}$ & $\begin{array}{c}
-\frac{1}{3}(e^7+e^5)\\
-2e^4
\end{array}$\\
\specialrule{1pt}{0pt}{0pt}
$147D$ & $\begin{array}{c}
-2e^{12}-6e^{23}-3e^{25}\\
+e^{34}-2e^{36}-e^{56}
\end{array}$
& $\begin{array}{c}
-e^{123}-2e^{134}-e^{136}+e^{145}\\
-e^{246}-e^{356}
\end{array}$ & $\begin{array}{c}
\frac{1}{7}(e^7-2e^6)\\
-e^4
\end{array}$\\
\specialrule{1pt}{0pt}{0pt}
$147E(\lambda)$ & $e^{13}-e^{26}-e^{34}+e^{45}$ & $-e^{123}-2e^{125}-e^{146}+e^{245}+e^{356}$ & $\frac{3}{1-\lambda}e^7+7e^6$ \\
\specialrule{1pt}{0pt}{0pt}
$147E_1(\lambda)$ & $e^{13}+e^{26}+e^{34}-e^{45}$ & $e^{123}+2e^{125}+e^{146}+e^{245}+e^{356}$ & $-\frac{1}{2}e^7-e^6$ \\
\specialrule{1pt}{0pt}{0pt}
$157$ & $\begin{array}{c}
		2e^{12}+e^{15}+e^{26}\\
		+e^{34}+e^{56}
	\end{array}$
& $\begin{array}{c}
e^{124}+e^{135}-e^{145}-e^{146}\\
-e^{235}-e^{236}-e^{245}
\end{array}$ & $e^7-2e^3$\\
\specialrule{1pt}{0pt}{0pt}
$247A$ & $e^{12}+e^{34}+e^{56}$
& $e^{135}-e^{146}-e^{236}-e^{245}$ & $e^7$\\
\specialrule{1pt}{0pt}{0pt}
$247B$ & $\begin{array}{c}
		e^{12}+e^{26}-e^{35}\\
		+e^{45}+e^{46}
	\end{array}$
& $\begin{array}{c}
-e^{123}+e^{124}-2e^{125}-e^{126}+e^{134}\\
-e^{135}-e^{136}-e^{145}+e^{156}-e^{235}\\
-e^{236}-e^{245}
\end{array}$ & $\begin{array}{c}
e^7+e^6\\
+3e^5+4e^4
\end{array}$\\
\specialrule{1pt}{0pt}{0pt}
$247C$ & $e^{13}+e^{25}+e^{47}$
& $e^{124}-e^{157}+e^{237}+e^{345}$ & $e^6+e^4$\\
\specialrule{1pt}{0pt}{0pt}
$247D$ & $e^{12}-e^{34}-e^{56}$
& $e^{136}+e^{145}-e^{235}+e^{246}$ & $e^7+e^4$\\
\specialrule{1pt}{0pt}{0pt}
$247F$ & $\begin{array}{c}
		e^{15}+e^{16}-e^{25}\\
		-e^{34}+e^{56}
	\end{array}$
& $\begin{array}{c}
e^{123}-e^{124}+e^{145}+2e^{235}\\
+e^{236}+e^{246}-e^{356}
\end{array}$ & $3e^7-e^4$\\
\specialrule{1pt}{0pt}{0pt}
$247F_1$ & $-e^{13}-e^{16}+\frac{3}{2}e^{23}+e^{26}-e^{45}$ & $e^{125}+e^{246}-e^{134}+e^{356}$ & $e^7-e^5$\\
\specialrule{1pt}{0pt}{0pt}
$247I$ & $e^{15}-\frac{1}{2}e^{16}-e^{23}-\frac{1}{2}(e^{45}-e^{46})$ & 
	$e^{126}-e^{134}+e^{245}-e^{356}$ & $e^7+e^4$\\
\specialrule{1pt}{0pt}{0pt}
$247J$ & $e^{13}-e^{25}-e^{46}$
& $\begin{array}{c}
-e^{124}+e^{126}-e^{145}+e^{236}\\
+2e^{345}+e^{356}
\end{array}$ & $e^7+6e^5$\\
\specialrule{1pt}{0pt}{0pt}
$247L$ & $e^{13}-e^{24}+e^{57}$ & $e^{125}+e^{345}+e^{147}+e^{237}$ & $\frac{1}{2}e^6$\\
\bottomrule[1pt]
\end{tabu}}
\end{table}

\begin{table}[H]
\centering
\caption{Purely coclosed $\Gtwo$-structures on indecomposable 3-step NLAs $-$ 2}\label{table:2a}
\vspace{0.25 cm}
{\tabulinesep=1.2mm
\begin{tabu}{l|c|c|c}
\toprule[1.5pt]
NLA & $\omega$ & $\psi_-$ & $\eta$\\
\specialrule{1pt}{0pt}{0pt}
$247M$ & $e^{15}+e^{23}+e^{35}-2e^{46}$ & 
$\begin{array}{c}
-2e^{124}-2e^{126}-e^{134}-2e^{136}-e^{145}\\
+e^{156}+e^{234}+e^{236}+e^{245}+e^{345}
\end{array}$ & $\begin{array}{c}
e^7+2e^5\\-2e^4
\end{array}$\\
\specialrule{1pt}{0pt}{0pt}
$247N$ & $-e^{14}-e^{25}-e^{37}-\frac{1}{2}(e^{17}+e^{34})$ & $-e^{123}+e^{157}-e^{247}+e^{345}$ & $e^6$\\
\specialrule{1pt}{0pt}{0pt}
$247O$ & $\begin{array}{c}
e^{12}+e^{13}+e^{15}-e^{24}\\
-\frac{1}{2}(e^{27}-e^{45})+e^{57}
\end{array}$ & $e^{125}+e^{345}+e^{147}+e^{237}+e^{245}-e^{357}$ & $\frac{5}{4}e^6$\\
\specialrule{1pt}{0pt}{0pt}
$247P$ & $e^{17}+e^{23}+e^{45}$ & $e^{125}+e^{247}+e^{134}-e^{357}$ & $-2e^6$\\
\specialrule{1pt}{0pt}{0pt}
$247P_1$ & $\begin{array}{c}
\frac{3}{2}(e^{15}+e^{47})+e^{17}+e^{23}\\
+e^{25}-e^{34}+\frac{9}{2}e^{45}
\end{array}$ & $e^{125}+e^{247}+e^{134}-e^{357}$ & $e^6-\frac{3}{2}e^5$\\
\specialrule{1pt}{0pt}{0pt}
$247Q$ & $e^{14}+e^{23}+e^{57}$ & $e^{125}-e^{137}+e^{247}+e^{345}$ & $-e^6$\\
\specialrule{1pt}{0pt}{0pt}
$257A$ & $-(e^{12}+e^{34}-e^{56})$ & $e^{136}-e^{145}-e^{235}-e^{246}$ & $e^7$\\
\specialrule{1pt}{0pt}{0pt}
$257C$ & $\begin{array}{c}
		e^{13}-e^{14}+e^{25}-e^{26}\\
		-e^{36}-e^{45}+3e^{46}
	\end{array}$ & $e^{124}-e^{156}+e^{236}+e^{345}$ & $2e^7-e^6-2e^3$\\
	\specialrule{1pt}{0pt}{0pt}
$257F$ & $-2e^{15}-e^{16}+3e^{24}+e^{35}+e^{56}$
& $\begin{array}{c}
e^{134}+e^{236}-e^{456}\\
-e^{345}+e^{126}+e^{235}
\end{array}$ & $-4e^7$\\
\specialrule{1pt}{0pt}{0pt}
$257I$ & $e^{12}-e^{35}-e^{45}+e^{46}$
& $e^{134}+e^{156}-e^{236}-e^{245}-e^{246}$ & $e^7-e^3$\\
\specialrule{1pt}{0pt}{0pt}
$257J$ & $e^{13}+e^{25}+e^{46}$
& $-e^{124}+e^{156}-e^{236}-e^{345}$ & $e^7+e^3$\\
\specialrule{1pt}{0pt}{0pt}
$257J_1$ & $\begin{array}{c}
e^{15}-e^{46}+\frac{1}{2}(e^{12}+e^{35})\\
+\frac{1}{2}(-e^{14}-e^{34}+e^{56})
\end{array}$
& $e^{124}+e^{136}-e^{234}-e^{256}-e^{345}$ & $\begin{array}{c}
-\frac{1}{8}e^7+\frac{1}{4}e^6\\[1.5pt]
-\frac{1}{8}e^3\end{array}$\\
\specialrule{1pt}{0pt}{0pt}
$357A$ & $-e^{13}+2e^{14}+e^{25}-2e^{36}+e^{46}$
& $e^{123}-e^{246}+e^{156}-e^{345}$ & $e^7+e^5$\\
\bottomrule[1pt]
\end{tabu}}
\end{table}
}

%

%

For the families $147E(\lambda)$ and $147E_1(\lambda)$ we cannot use the usual Sage worksheets, since the DGA package does not admit
parameters in the definition of the differential. Instead we work with the following worksheet, inspired by \cite{Angella}, that implements an older way to deal with DGAs, less user-friendly, but that allows variables in the definition of the differential. For instance, $147E(\lambda)$ is implemented as follows:

{\tiny
\noindent \hspace{1mm} \begin{verb}
E = ExteriorAlgebra(SR,'x',8)
\end{verb}

\noindent \hspace{1mm} \begin{verb}
l=var('l')
\end{verb}

\noindent \hspace{1mm} \begin{verb}
str_eq={(1,2):E.gens()[4],(2,3):E.gens()[5],(1,3):-E.gens()[6],(1,5):-E.gens()[7],(2,6):l*E.gens()[7],
\end{verb}

\begin{verb}
(3,4):(1-l)*E.gens()[7]}
\end{verb}

\noindent \hspace{1mm} \begin{verb}
d=E.coboundary(str_eq); d
\end{verb}

\noindent \hspace{1mm} \begin{verb}
print([d(b) for b in E.gens( )])
\end{verb}


\noindent \hspace{1mm} \begin{verb}
omega=E.gens()[1]*E.gens()[3]-E.gens()[2]*E.gens()[6]-E.gens()[3]*E.gens()[4]+E.gens()[4]*E.gens()[5]
\end{verb}

\noindent \hspace{1mm} \begin{verb}
psi=-E.gens()[1]*E.gens()[2]*E.gens()[3]-2*E.gens()[1]*E.gens()[2]*E.gens()[5]
\end{verb}

\begin{verb}
-E.gens()[1]*E.gens()[4]*E.gens()[6]+E.gens()[2]*E.gens()[4]*E.gens()[5]
\end{verb}

\begin{verb}
+E.gens()[3]*E.gens()[5]*E.gens()[6]
\end{verb}

\noindent \hspace{1mm} \begin{verb}
eta=(1/(l-1))*E.gens()[7]- E.gens()[6]
\end{verb}

\noindent \hspace{1mm} \begin{verb}
omega*psi
\end{verb}

\noindent \hspace{1mm} \begin{verb}
psi*psiplus-(2/3)*omega^3
\end{verb}

\noindent \hspace{1mm} \begin{verb}
d(psi)
\end{verb}

\noindent \hspace{1mm} \begin{verb}
omega*d(omega)-psi*d(eta)
\end{verb}

\noindent \hspace{1mm} \begin{verb}
omega^2*d(eta)+2*psiplus*d(omega)
\end{verb}

}

\subsection{Indecomposable 4-step nilpotent Lie algebras}

We deal next with 7D indecomposable 4-step nilpotent Lie algebras. According to Gong \cite{Gong} there are 43 of them, listed in Tables \ref{7d-ind-4step-1} and \ref{7d-ind-4step-2}. Four of them ($1357M(\lambda)$, $1357N(\lambda)$, $1357QRS_1(\lambda)$ and $1357S(\lambda)$) depend on a real parameter $\lambda$; there are conditions on the value of the parameter, which are included in Appendix \ref{appendix}.

We have the following result:

\begin{theorem}\label{4step-indecomposable-pure}
All $7$-dimensional indecomposable $4$-step nilpotent Lie algebra admit a purely coclosed $\Gtwo$-structure, except for: $1357E$, $1457A$, $1457B$ and $1357N(-2)$.
\end{theorem}

In the Lie algebras $1357M(\lambda)$, $1357N(\lambda)$ ($\lambda\neq -2$) and $1357QRS_1(\lambda)$, $\psi_-$ depends on the
parameter $\lambda$. To check that $\psi_-\in \Lambda_-(\lie{h}^*)$, we use the commands of subsection \ref{subsec:sage},
adding an extra variable. As we are not allowed to set its degree as $0$, we set any even degree, which makes no 
difference to the computer to yield the result. This is the reason for setting its degree equal to $2$ in the worksheet.

{\footnotesize 
\begin{table}[H]
\centering
\caption{Purely coclosed $\Gtwo$-structures on indecomposable 4-step NLAs $-$ 1}\label{table:4}
\vspace{0.25 cm}
{\tabulinesep=1.2mm
\begin{tabu}{l|c|c|c}
\toprule[1.5pt]
NLA & $\omega$ & $\psi_-$ & $\eta$\\
\specialrule{1pt}{0pt}{0pt}
$1357A$ & $-e^{12}+e^{34}-e^{56}$ & $e^{135}+e^{136}-e^{145}+e^{146}+e^{235}+e^{246}$ & $e^7-2e^4$\\
\specialrule{1pt}{0pt}{0pt}
$1357B$ & $-e^{13}-e^{24}+\frac{1}{4}e^{45}-e^{56}$ & $2e^{126}+e^{145}+\frac{1}{2}e^{156}-e^{235}-\frac{1}{2}e^{346}$ & $e^7-4e^4$\\
\specialrule{1pt}{0pt}{0pt}
$1357C$ & $\begin{array}{c}
		-\frac{1}{2}e^{12}+2e^{15}-4e^{25}-6e^{35}\\[1.5pt]
		-\frac{1}{2}e^{36}+28e^{45}+e^{46}-2e^{56}
	\end{array}$
& $\begin{array}{c}
-2e^{124}-3e^{134}+4e^{145}+2e^{156}\\
-4e^{235}-2e^{246}-2e^{346}
\end{array}$ & $-2e^7+4e^5$\\
\specialrule{1pt}{0pt}{0pt}
$1357D$ & $e^{12}+e^{34}+e^{56}$ & $e^{136}+e^{145}-e^{146}-e^{235}
+e^{236}+e^{246}$ & $e^7-2e^3$\\
\specialrule{1pt}{0pt}{0pt}
$1357F$ & $\begin{array}{c}
		e^{12}+e^{13}+e^{35}\\
		-e^{46}+e^{56}
	\end{array}$ & $\begin{array}{c}
		e^{123}-e^{124}+e^{125}-2e^{126}\\
		+e^{136}+e^{145}-e^{234}+e^{236}\\
		+e^{245}-e^{246}-e^{256}
	\end{array}$ & $e^7+2e^5$\\
\specialrule{1pt}{0pt}{0pt}
$1357F_1$ & $e^{12}+e^{35}-e^{46}+e^{56}$ & $\begin{array}{c}
		-e^{134}+e^{136}+e^{145}-e^{146}\\
		-e^{234}-e^{245}-e^{246}-e^{256}
		\end{array}$ & $e^7+e^6$\\
\specialrule{1pt}{0pt}{0pt}
$1357G$ & $e^{12}-e^{34}-2e^{56}$ & $\begin{array}{c}
		e^{136}+e^{145}-e^{146}-e^{235}\\
		-e^{236}-e^{245}+\frac{9}{4}e^{246}
		\end{array}$ & $e^7+\frac{5}{2}e^3$\\
\specialrule{1pt}{0pt}{0pt}
$1357H$ & $-e^{12}+2e^{34}-e^{56}$ & $\begin{array}{c}
e^{135}+e^{145}+e^{146}\\
+e^{235}-e^{236}+2e^{245}
\end{array}$ & $e^7-\frac{1}{2}e^3$\\
\bottomrule[1pt]
\end{tabu}}
\end{table}
}

{\footnotesize 
\begin{table}[H]
\centering
\caption{Purely coclosed $\Gtwo$-structures on indecomposable 4-step NLAs $-$ 2}\label{table:4a}
\vspace{0.25 cm}
{\tabulinesep=1.2mm
\begin{tabu}{l|c|c|c}
\toprule[1.5pt]
NLA & $\omega$ & $\psi_-$ & $\eta$\\
\specialrule{1pt}{0pt}{0pt}
$1357I$ & $\begin{array}{c}
e^{15}+e^{23}+4e^{24}\\
+e^{36}+e^{56}
\end{array}$ & 
$e^{123}-e^{146}+e^{256}+e^{345}$ & $e^7+e^6+2e^5$\\
\specialrule{1pt}{0pt}{0pt}
$1357J$ & $\begin{array}{c}
-e^{13}+2e^{23}+e^{25}\\
-e^{46}+e^{56}
\end{array}$ & 
$\begin{array}{c}
2e^{123}+3e^{124}-2e^{125}+e^{136}\\
-e^{256}+e^{146}-e^{345}
\end{array}$ & $\begin{array}{c}
-\frac{2}{3}e^7+\frac{2}{3}e^6\\[1.5pt]
+\frac{4}{3}e^5+2e^3
\end{array}$\\
\specialrule{1pt}{0pt}{0pt}
$1357L$ & $e^{12}+e^{34}+3e^{56}$
& $\begin{array}{c}
2e^{135}-\frac{2}{3}e^{136}-e^{146}-e^{236}\\
-3e^{245}+e^{246}
\end{array}$ & $3e^7-3e^6-\frac{1}{2}e^3$\\
\specialrule{1pt}{0pt}{0pt}
\hspace{-0.2 cm}$\begin{array}{l}
1357M(\lambda)\\
\lambda<-1
\end{array}$ & $e^{12}-e^{34}-e^{56}$ & $\begin{array}{c}
e^{135}-e^{146}-e^{235}-\lambda e^{236}\\
-(\lambda+1)e^{245}+e^{246}
\end{array}$ & $\begin{array}{c}
e^7\\
-\frac{\lambda^3-\lambda-1}{\lambda(\lambda+1)}e^3
\end{array}$ \\
\specialrule{1pt}{0pt}{0pt}
$1357M(-1)$ & $-e^{12}-e^{34}+e^{56}$ & $e^{235}+2e^{145}-e^{136}+e^{246}$ & $e^7-\frac{3}{2}e^3$ \\
\specialrule{1pt}{0pt}{0pt}
\hspace{-0.2 cm}$\begin{array}{l}
1357M(\lambda)\\
-1<\lambda<0
\end{array}$ & $-e^{12}-e^{34}+e^{56}$ & $\begin{array}{c}
e^{135}+e^{146}+e^{235}-\lambda e^{236}\\
-(\lambda+1)e^{245}+e^{246}
\end{array}$ & $\begin{array}{c}
-e^7\\
-\frac{\lambda^3-\lambda-1}{\lambda(\lambda+1)}e^3
\end{array}$ \\
\specialrule{1pt}{0pt}{0pt}
\hspace{-0.2 cm}$\begin{array}{l}
1357M(\lambda)\\
\lambda>0
\end{array}$ & $e^{12}+e^{34}+e^{56}$ & $\begin{array}{c}
e^{135}-e^{146}-e^{235}-\lambda e^{236}\\
-(\lambda+1)e^{245}+e^{246}
\end{array}$ & $\begin{array}{c}
e^7\\
+\frac{\lambda^3-\lambda-1}{\lambda(\lambda+1)}e^3
\end{array}$ \\
\specialrule{1pt}{0pt}{0pt}
\hspace{-0.2 cm}$\begin{array}{l}
1357N(\lambda)\\
\lambda<-2
\end{array}$ & $\begin{array}{c}
\lambda e^{12}+e^{14}\\[1.5pt]
+\left(\lambda+\frac{5}{2}\right)e^{23}+e^{36}\\[1.5pt]
+(\lambda+1)e^{25}+e^{56}\end{array}$ & $\begin{array}{c}
-2(\lambda-2)(e^{126}-e^{134}-e^{146}\\
+e^{245})+4e^{135}+2e^{156}\\
-\lambda\sqrt{-(\lambda+2)}e^{246}\\
-2\lambda e^{236}+2e^{346}
\end{array}$ & $\begin{array}{c}
\sqrt{-\lambda-2}\left(-\frac{1}{\lambda+2}e^7\right.\\[1.5pt]
-\frac{3\lambda-10}{3\lambda}e^6\\[1.5pt]
\left.+\frac{7}{3}e^3\right)\end{array}$\\
\specialrule{1pt}{0pt}{0pt}
\hspace{-0.2 cm}$\begin{array}{l}
1357N(\lambda)\\
-2<\lambda<0
\end{array}$ & $\begin{array}{c}
(\lambda+2)^2e^{12}\\
+(\lambda+2)e^{34})\\
-e^{56}\end{array}$ & $\begin{array}{c}
-\frac{5}{2}e^{123}-e^{125}+e^{126}+2e^{135}\\
+e^{136}+e^{156}+e^{234}-\lambda e^{236}\\
-(\lambda+2)e^{245}+e^{346}
\end{array}$ & $\begin{array}{c}
\frac{3}{2(\lambda+2)}(e^7-e^5)\\[1.5pt]
-2\frac{3\lambda^2 + 8\lambda + 8}{\lambda(\lambda+2)}e^6\\[1.5pt]
-\frac{24\lambda^2 + 64\lambda + 61}{2(\lambda+2)}e^3\end{array}$\\
\specialrule{1pt}{0pt}{0pt}
$1357N(0)$ & $e^{14}+e^{23}+e^{56}$ & $\begin{array}{c}
e^{123}-e^{125}+2e^{126}-2 e^{135}\\
+e^{136}-e^{156}+2e^{245}+e^{346}
\end{array}$ & $\begin{array}{c}
e^7-\frac{16}{23}e^6\\[1.5pt]
+e^5-2e^3
\end{array}$\\
\specialrule{1pt}{0pt}{0pt}
\hspace{-0.2 cm}$\begin{array}{l}
1357N(\lambda)\\
\lambda>0
\end{array}$ & $\begin{array}{c}
2\lambda e^{12}+(\lambda+2)e^{14}\\
+\lambda e^{34}+e^{56}\end{array}$ & $\begin{array}{c}
-2\lambda(e^{126}-e^{134}-e^{146}-e^{236})\\
-4e^{135}-2e^{156}+2(\lambda+2)e^{245}\\
-e^{246}-2e^{346}
\end{array}$ & $\begin{array}{c}
\frac{4\lambda^4+15\lambda^3 + 32\lambda^2 + 68\lambda + 64}{\lambda(4\lambda^2 + 15\lambda + 16)}e^6\\[1.5pt]
+\frac{4\lambda^3 + 31\lambda^2 + 84\lambda + 64}{4\lambda^2 + 15\lambda + 16}e^3\\[1.5pt]
-e^7
\end{array} $\\
\specialrule{1pt}{0pt}{0pt}
$1357O$ & $e^{12}-e^{34}-e^{56}$
& $\begin{array}{c}
e^{136}+\frac{5}{4}e^{145}-\frac{1}{2}e^{146}\\
-e^{235}-\frac{1}{2}e^{245}+e^{246}
\end{array}$ & $e^7-\frac{13}{4}e^3$\\
\specialrule{1pt}{0pt}{0pt}
$1357P$ & $e^{12}+e^{34}+e^{56}$
& $\begin{array}{c}
-e^{135}+e^{146}-e^{235}\\
+e^{236}+2e^{245}+e^{246}
\end{array}$ & $e^7-2e^3$\\
\bottomrule[1pt]
\end{tabu}}
\end{table}
}

{\footnotesize 
\begin{table}[H]
\centering
\caption{Purely coclosed $\Gtwo$-structures on indecomposable 4-step NLAs $-$ 3}\label{table:5}
\vspace{0.25 cm}
{\tabulinesep=1.2mm
\begin{tabu}{l|c|c|c}
\toprule[1.5pt]
NLA & $\omega$ & $\psi_-$ & $\eta$\\
\specialrule{1pt}{0pt}{0pt}
$1357P_1$ & $e^{12}+2e^{34}+2e^{56}$
& $e^{145}+e^{136}+e^{235}-e^{246}$ & $-2e^7-e^3$\\
\specialrule{1pt}{0pt}{0pt}
$1357Q$ & $e^{12}+e^{34}+e^{56}$ & $\begin{array}{c}
-e^{135}+e^{146}-e^{235}\\
+e^{236}+e^{245}+e^{246}
\end{array}$ & $e^7-2e^3$\\
\specialrule{1pt}{0pt}{0pt}
$1357Q_1$ & $-e^{12}-e^{34}+2e^{56}$ & $\begin{array}{c}
-e^{135}+\frac{17}{4}e^{145}-e^{146}+e^{235}\\
-e^{236}-e^{245}+\frac{9}{4}e^{246}
\end{array}$ & $e^7+\frac{153}{8}e^3$\\
\specialrule{1pt}{0pt}{0pt}
\hspace{-0.2 cm}$\begin{array}{l}
1357QRS_1(\lambda)\\
\lambda>0
\end{array}$ & $e^{12}+\lambda e^{34}+\lambda e^{56}$ & $
\begin{array}{c}
-e^{135}+2e^{145}+\lambda e^{146}+e^{235}\\
+\lambda e^{236}+e^{245}-\lambda e^{246}
\end{array}$ & $\lambda e^7+\frac{1}{2}(3\lambda+1)e^3$ \\
\specialrule{1pt}{0pt}{0pt}
\hspace{-0.2 cm}$\begin{array}{l}
1357QRS_1(\lambda)\\
\lambda<0
\end{array}$ & $-e^{12}+\lambda e^{34}-\lambda e^{56}$ & $
\begin{array}{c}
-e^{135}+2e^{145}+\lambda e^{146}+e^{235}\\
+\lambda e^{236}+e^{245}-\lambda e^{246}
\end{array}$ & $-\lambda e^7+\frac{1}{2}(3\lambda+1)e^3$ \\
\specialrule{1pt}{0pt}{0pt}
$1357R$ & $-e^{12}-e^{34}+e^{56}$ & $\begin{array}{c}
-e^{135}-e^{136}+e^{145}\\
+e^{235}+e^{245}+e^{246}
\end{array}$ & $e^7+e^3$\\
\specialrule{1pt}{0pt}{0pt}
$1357S(0)$ & $\begin{array}{c}
-e^{12}-3e^{34}+e^{56}\\
-e^{35}+2e^{46}
\end{array}$ & $-e^{145}+e^{136}-e^{245}-e^{235}-2e^{246}$ & $e^7-\frac{19}{7}e^3$ \\
\specialrule{1pt}{0pt}{0pt}
\hspace{-0.2 cm}$\begin{array}{l}
1357S(\lambda)\\
\lambda\neq 0
\end{array}$ & $-e^{12}-e^{34}+e^{56}$ & $e^{136}-e^{145}-e^{235}-e^{245}-e^{246}$ & $\frac{1}{\lambda}e^7-\left(\frac{1}{\lambda}+\frac{4}{3}\right)e^3$ \\
\specialrule{1pt}{0pt}{0pt}
$2357A$ & $\begin{array}{c}
e^{13}-e^{24}+e^{57}\\
+2e^{14}-2e^{23}
\end{array}$ & $-e^{125}-e^{237}-e^{147}-e^{345}$ & $\frac{9}{2}e^6-\frac{9}{4}e^5$\\
\specialrule{1pt}{0pt}{0pt}
$2357B$ & $e^{13}-e^{24}+e^{57}$ & $-e^{125}-e^{147}-e^{237}-e^{345}$ & $-3e^6$\\
\specialrule{1pt}{0pt}{0pt}
$2357C$ & $e^{13}-e^{24}+e^{57}$ & $-e^{125}-e^{147}-e^{237}-e^{345}$ & $3e^6$\\
\specialrule{1pt}{0pt}{0pt}
$2357D$ & $\begin{array}{c}
3e^{13}-e^{24}+e^{57}+e^{12}\\
-e^{34}+e^{14}-e^{23}
\end{array}$ & 
$-e^{125}-e^{237}-e^{147}-e^{345}$ & $-e^7+3e^6-2e^4$\\
\specialrule{1pt}{0pt}{0pt}
$2357D_1$ & $\begin{array}{c}
e^{13}+e^{24}-e^{57}\\
+\frac{1}{2}(e^{15}-e^{37})
\end{array}$ & 
$e^{125}+e^{147}-e^{237}-e^{345}$ & $3e^6-\frac{1}{2}e^4$\\
\specialrule{1pt}{0pt}{0pt}
$2457A$ & $e^{15}+e^{24}+e^{36}$
& $e^{123}-e^{146}+e^{256}+e^{345}$ & $e^7$\\
\specialrule{1pt}{0pt}{0pt}
$2457B$ & $\begin{array}{c}
e^{15}-e^{24}-e^{36}\\
+e^{13}-e^{56}
\end{array}$ & $e^{123}-e^{146}-e^{256}-e^{345}$ & $e^7-e^6+e^4$\\
\specialrule{1pt}{0pt}{0pt}
$2457C$ & $e^{15}-e^{24}+e^{36}$ & $-e^{123}+e^{146}+e^{256}+2e^{345}$ & $e^7$\\
\specialrule{1pt}{0pt}{0pt}
$2457D$ & $\begin{array}{c}
2e^{15}-e^{23}+2e^{26}\\
+e^{34}+e^{36}+3e^{46}
\end{array}$
& $\begin{array}{c}
e^{124}+e^{126}+e^{136}\\
+e^{245}+e^{256}+e^{345}
\end{array}$ & $10e^7+6e^3$\\
\specialrule{1pt}{0pt}{0pt}
$2457E$ & $e^{15}-e^{23}+e^{47}$
& $e^{124}+e^{137}+e^{257}+e^{345}$ & $-e^6+e^4$\\
\specialrule{1pt}{0pt}{0pt}
$2457F$ & $-e^{14}+e^{23}-2e^{46}-e^{56}$
& $e^{124}+e^{125}+e^{136}-e^{256}-e^{345}$ & $-e^7$\\
\specialrule{1pt}{0pt}{0pt}
$2457G$ & $2e^{15}-e^{23}+e^{47}$
& $e^{124}+e^{137}+e^{257}+e^{345}$ & $-2e^6$\\
\bottomrule[1pt]
\end{tabu}}
\end{table}
}

{\footnotesize 
\begin{table}[H]
\centering
\caption{Purely coclosed $\Gtwo$-structures on indecomposable 4-step NLAs $-$ 4}\label{table:5a}
\vspace{0.25 cm}
{\tabulinesep=1.2mm
\begin{tabu}{l|c|c|c}
\toprule[1.5pt]
NLA & $\omega$ & $\psi_-$ & $\eta$\\

\specialrule{1pt}{0pt}{0pt}
$2457H$ & $e^{12}-e^{35}+e^{46}$
& $e^{136}+e^{145}+e^{234}+e^{256}$ & $-e^7+e^4$\\
\specialrule{1pt}{0pt}{0pt}
$2457I$ & $\begin{array}{c}
e^{15}+e^{23}-e^{46}-e^{56}\\
+\frac{1}{2}(e^{13}+e^{25})
\end{array} $
& $e^{124}+e^{125}+e^{136}-e^{256}-e^{345}$ & $\frac{1}{2}e^7+e^4$\\
\specialrule{1pt}{0pt}{0pt}
$2457J$ & $e^{15}+e^{23}-e^{46}-e^{56}$
& $e^{124}+e^{125}+e^{136}-e^{256}-e^{345}$ & $-e^7-e^4$\\
\specialrule{1pt}{0pt}{0pt}
$2457K$ & $-e^{12}+e^{35}-e^{46}$ & $
e^{136}+e^{145}+e^{234}+e^{256}$ & $-e^7+e^4+e^3$\\
\specialrule{1pt}{0pt}{0pt}
$2457L$ & $e^{15}+e^{24}+e^{36}$ & $e^{123}+e^{134}-e^{146}+e^{235}+3e^{256}+2e^{345}$ & $18e^7$\\
\specialrule{1pt}{0pt}{0pt}
$2457L_1$ & $-e^{14}+2e^{25}+e^{37}$ & $-e^{123}+e^{135}+e^{157}+e^{234}+e^{247}-2e^{345}$ & $\frac{2}{3}e^6$\\
\specialrule{1pt}{0pt}{0pt}
$2457M$ & $2e^{15}+e^{24}+e^{37}$ & $e^{123}-e^{147}+e^{257}+e^{345}$ & $-\frac{2}{3}e^6$\\
\bottomrule[1pt]
\end{tabu}}
\end{table}
}

Summarizing, the proofs of Theorems \ref{2step-indecomposable-pure}, \ref{3step-indecomposable-pure} 
and \ref{4step-indecomposable-pure} have two parts:
\begin{itemize}
\item Using Theorem \ref{thm:construction}, we exhibit an explicit purely coclosed $\Gtwo$-structure on each indecomposable NLA which is not mentioned in the statements of the theorems. These are given in Tables \ref{table:2}--\ref{table:5a}. 
For the convenience of the reader, we provide SageMath worksheets \cite{worksheets} which can be used as in Subsection \ref{subsec:sage}.
\item We use the obstructions of Section \ref{sec:coclosed} to prove that the Lie algebras in the statement of 
Theorems \ref{2step-indecomposable-pure}--\ref{4step-indecomposable-pure} do not admit any coclosed $\Gtwo$-structure.
\end{itemize}

	As a consequence of the results of this section, we get the following result:

\begin{corollary}
Every $7$-dimensional indecomposable nilpotent Lie algebra of nilpotency step $\leq 4$ admitting a coclosed $\Gtwo$-structure also admits a purely coclosed one.
\end{corollary}

We expect that {\em all} indecomposable 7-dimensional nilpotent Lie algebras carrying a coclosed $\Gtwo$-structure also admit a purely coclosed one. Indeed, besides the {\em ad hoc} argument which allows to rule out $\fn_2$ from the list of Lie algebras admitting a purely coclosed $\Gtwo$-structure (see \cite[Corollary 4.3]{BMR}), no general obstruction to the existence of a purely coclosed  $\Gtwo$-structure is known. However, it is difficult to deduce the existence of a purely coclosed $\Gtwo$-structure directly from the existence of a coclosed one, under the indecomposability assumption.

A final remark concerns the existence of {\em exact} purely coclosed $\Gtwo$-structures: these are purely coclosed $\Gtwo$-structures whose 4-form $\phi=\ast_\varphi\varphi$ is exact. Notice that there are examples of exact coclosed $\Gtwo$-structures: it suffices to consider a nearly parallel $\Gtwo$-structure, for which $d\varphi=\ast_\varphi\varphi=\phi$. All the purely coclosed $\Gtwo$-structures we constructed in this paper turn out not to be exact. The same question has been asked in the context of closed $\Gtwo$-structures.  A $\Gtwo$-structure is {\em exact} if its defining 3-form $\varphi$ is exact. It was recently proved that no compact quotient of a Lie group by a discrete subgroup admits an exact $\Gtwo$-structure which is left-invariant (see \cite[Theorem 1.1]{FMMR}).

\begin{conjecture}
No compact quotient of a Lie group by a discrete subgroup admits an exact purely coclosed $\Gtwo$-structure.
\end{conjecture}

\section{7D nilpotent Lie algebras with no coclosed $\Gtwo$-structures}\label{section:nococlosed}

We prove that the following algebras do not have coclosed $G_2$-structures:
\begin{align*}
& 27A, 27B \\
& 247E, 247E_1, 247G, 247H, 247H_1, 247K, 247R, 247R_1\\
& 257B, 257D, 257E, 257G, 257H, 257K, 257L\\
& 357B, 357C\\
& 1357E, 1357N(-2) \\
& 1457A, 1457B\,.
\end{align*}

In order to prove this, we use three different types of obstructions. We list in order:

%
%

\begin{itemize}
\item We use Corollary \ref{cor:obs3} to cover the families $27$, $247$, $257$, as well as $1357E$. 
For instance, for $27A$, we have the following worksheet

\smallskip

{\tiny
\noindent \hspace{-2mm} \begin{verb}
A.<x1,x2,x3,x4,x5,x6,x7> = GradedCommutativeAlgebra(QQ)
\end{verb}

\noindent \hspace{-2mm} \begin{verb}
M=A.cdg_algebra({x6:x1*x2,x7: x1*x5 +x2*x3})
\end{verb}

\noindent \hspace{-2mm} \begin{verb}
M.inject_variables()
\end{verb}

\noindent \hspace{-2mm} \begin{verb}
M.cohomology(4)
\end{verb}

\smallskip

\noindent \hspace{-2mm} \begin{verb}
Defining x1, x2, x3, x4, x5, x6, x7
\end{verb}

\noindent \hspace{-2mm} \begin{verb}
Free module generated by {[x1*x2*x3*x6],[x1*x2*x4*x6],[x1*x3*x4*x6],[x2*x3*x4*x6],[x1*x2*x5*x6], 
\end{verb}

\noindent \hspace{-2mm} \begin{verb}
[x1*x4*x5*x6],[x2*x4*x5*x6],[x3*x4*x5*x6-x2*x3*x5*x7],[x1*x3*x4*x7],[x2*x3*x4*x7],[x1*x3*x5*x7], 
\end{verb}

\noindent \hspace{-2mm} \begin{verb}
[x1*x4*x5*x7],[x2*x4*x5*x7],[x3*x4*x5*x7],[x1*x3*x6*x7],[x1*x5*x6*x7]} over Rational Field
\end{verb}
}
\smallskip

\noindent Recall that in the worksheet we write $xi$ for $e^i$. 
We use $X=e_6$, $Y=e_7$, and the space $U=\la e^{13} \ra$. Using
the list of representatives of the cohomology classes, we see that $\imath_{e_6}\imath_{e_7} z_\alpha \in U$. Also $\Lambda^2 U=0$, as required.


\item We use Corollary \ref{cor:lem2} for the algebras 357B and 357C. In the Sage worksheet, we find the closed $4$-forms by 
computing the degree $4$ cohomology and the exact $4$-forms, explicitly. 


\item We use Proposition \ref{prop:third} for the algebras $1357N(-2)$, $1457A$, and $1457B$.
The Sage worksheets are self-commented and constructive. Bases of the spaces of closed $2$-forms $\beta$ and closed $3$-forms $\tau$
are found. Then the conditions $\beta\wedge d\beta,\tau\wedge de^j \in W$, $1\leq j\leq 6$, are checked.
The computation of $\lambda (\sum a_\alpha z_\alpha)$ is done using parameters and the implementation of the 
condition $\sum a_\alpha z_\alpha \wedge de^7\in W$ gives linear equations among $(a_\alpha)$ that reduces the expression of
$\lambda (\sum a_\alpha z_\alpha)$ to a square in each case.

\end{itemize}

\appendix

\section{7D nilpotent Lie alegbras of nilpotency step $\leq 4$}\label{appendix}

In this appendix we list 7-dimensional nilpotent Lie algebras of nilpotency step $\leq 4$ using differentials. The Chevalley-Eilenberg differential $d\colon\fg^*\to\Lambda^2\fg^*$ is dual to the bracket $[\cdot,\cdot]\colon \Lambda^2\fg\to\fg$. The list of Gong \cite{Gong} is given in terms of a nilpotent frame $\{e_1,\ldots,e_7\}$ of $\fg$; we give the structure constants in terms of the nilpotent coframe $\{-e^1,\ldots,-e^7\}$ of $\fg^*$. 

\begin{itemize}
\item In the family $147E(\lambda)$, two Lie algebras $147E(\lambda_1)$ and $147E(\lambda_2)$ are isomorphic if and only if $\frac{\left(1-\lambda_1+\lambda_1^2\right)^3}{\lambda_1^2(\lambda_1-1)^2}=\frac{\left(1-\lambda_2+\lambda_2^2\right)^3}{\lambda_2^2(\lambda_2-1)^2}$;
\item in the family $1357QRS_1(\lambda)$, two Lie algebras $1357QRS_1(\lambda_1)$ and $1357QRS_1(\lambda_2)$ are isomorphic if and only if $\lambda_1+\lambda_1^{-1}=\lambda_2+\lambda_2^{-1}$, that is, if and only if $\lambda_2=\lambda_1$ or $\lambda_2=\frac{1}{\lambda_1}$.
\end{itemize}


{\small 
\begin{table}[H]
\centering
\caption{7-dimensional decomposable nilpotent Lie algebras  $-$ 1}\label{table:7Ddec}
\vspace{0.25 cm}
{\tabulinesep=1.2mm
\begin{tabu}{c|c||c|c}
\toprule[1.5pt]
NLA & structure equations & NLA & structure equations \\
\specialrule{1pt}{0pt}{0pt}
$\fn_1$ & $(0,0,0,0,0,0,0)$ & $\fn_2$ & $(0,0,0,0,0,12,0)$\\
\specialrule{1pt}{0pt}{0pt}
$\fn_3$ & $(0,0,0,0,0,12+34,0)$ & $\fn_4$ & $(0,0,0,0,12,13,0)$ \\
\specialrule{1pt}{0pt}{0pt}
$\fn_5$ & $(0,0,0,0,12,34,0)$ & $\fn_6$ & $(0,0,0,0,13-24,14+23,0)$\\
\specialrule{1pt}{0pt}{0pt}
$\fn_7$ & $(0,0,0,0,12,14+23,0)$ & $\fn_8$ & $(0,0,0,0,12,14+25,0)$ \\
\specialrule{1pt}{0pt}{0pt}
$\fn_9$ & $(0,0,0,0,12,15+34,0)$ & $\fn_{10}$ & $(0,0,0,12,13,23,0)$\\
\specialrule{1pt}{0pt}{0pt}
$\fn_{11}$ & $(0,0,0,12,13,24,0)$ & $\fn_{12}$ & $(0,0,0,12,13,14+23,0)$ \\
\specialrule{1pt}{0pt}{0pt}
$\fn_{13}$ & $(0,0,0,12,23,14+35,0)$ & $\fn_{14}$ & $(0,0,0,12,23,14-35,0)$\\
\specialrule{1pt}{0pt}{0pt}
$\fn_{15}$ & $(0,0,0,12,13,14+35,0)$ & $\fn_{16}$ & $(0,0,0,12,14,15,0)$\\
\specialrule{1pt}{0pt}{0pt}
$\fn_{17}$ & $(0,0,0,12,14,15+24,0)$ & $\fn_{18}$ & $(0,0,0,12,14,15+24+23,0)$\\
\specialrule{1pt}{0pt}{0pt}
$\fn_{19}$ & $(0,0,0,12,14,15+23,0)$ & $\fn_{20}$ & $(0,0,0,12,14-23,15+34,0)$\\
\specialrule{1pt}{0pt}{0pt}
$\fn_{21}$ & $(0,0,12,13,23,14+25,0)$ & $\fn_{22}$ & $(0,0,12,13,23,14-25,0)$\\
\specialrule{1pt}{0pt}{0pt}
$\fn_{23}$ & $(0^2,12,13,23,14,0)$ & $\fn_{24}$ & $(0^2,12,13,14+23,15+24,0)$\\
\specialrule{1pt}{0pt}{0pt}
$\fg_1$ & $(0,0,0,0,12,15,0)$ & $\fg_2$ & $(0,0,0,0,23,34,36)$\\
\specialrule{1pt}{0pt}{0pt}
$\fg_3$ & $(0,0,0,12,13,14,0)$ & $\fg_4$ & $(0,0,0,12,14,24,0)$\\
\specialrule{1pt}{0pt}{0pt}
$\fg_5$ & $(0,0,12,13,14,23+15,0)$ & $\fg_6$ & $(0,0,12,13,14,15,0)$\\
\specialrule{1pt}{0pt}{0pt}
$\fg_7$ & $(0,0,12,13,14,34-25,0)$ & $\fg_8$ & $(0,0,12,13,14+23,34-25,0)$\\
\specialrule{1pt}{0pt}{0pt}
$\lie{l}_1$ & $(0,0,0,12,13-24,14+23,0)$ & $\lie{l}_2$ & $(0,0,0,12,14,13-24,0)$\\
\specialrule{1pt}{0pt}{0pt}
$\lie{l}_3$ & $(0,0,0,12,13+14,24,0)$ & & \\
\bottomrule[1pt]
\end{tabu}}
\end{table}
}



{\footnotesize 
\begin{table}[H]
\centering
\caption{7-dimensional indecomposable 2-step nilpotent Lie algebras}\label{7d-ind-2step}
\vspace{0.25 cm}
{\tabulinesep=1.2mm
\begin{tabu}{c|c|c||c|c|c}
\toprule[1.5pt]
NLA & structure equations & center & NLA & structure equations & center\\
\specialrule{1pt}{0pt}{0pt}
$17$ & $\left(0^6,12+34+56\right)$ & $\langle e_7\rangle$ & $27A$ & $\left(0^5,12,14+35\right)$ & $\langle e_6,e_7\rangle$\\
\specialrule{1pt}{0pt}{0pt}
$27B$ & $\left(0^5,12+34,15+23\right)$ & $\langle e_6,e_7\rangle$ & $37A$ & $\left(0^4,12, 23,24\right)$ & $\langle e_5,e_6,e_7\rangle$\\
\specialrule{1pt}{0pt}{0pt}
$37B$ & $\left(0^4,12,23,34\right)$ & $\langle e_5,e_6,e_7\rangle$ & $37B_1$ & $\left(0^4,12-34,13+24,14\right)$ & $\langle e_5,e_6,e_7\rangle$\\
\specialrule{1pt}{0pt}{0pt}
$37C$ & $\left(0^4,12+34,23,24\right)$ & $\langle e_5,e_6,e_7\rangle$ & $37D$ & $\left(0^4,12+34,13,24\right)$ & $\langle e_5,e_6,e_7\rangle$ \\
\specialrule{1pt}{0pt}{0pt}
$37D_1$ & $\left(0^4,12-34,13+24,14-23\right)$ & $\langle e_5,e_6,e_7\rangle$\\
\bottomrule[1pt]
\end{tabu}}
\end{table}
}


{\footnotesize 
\begin{table}[H]
\centering
\caption{7-dimensional indecomposable 3-step nilpotent Lie algebras $-$ 1}\label{7d-ind-3step}
\vspace{0.25 cm}
{\tabulinesep=1.2mm
\begin{tabu}{l|c|c}
\toprule[1.5pt]
NLA & structure equations & center\\
\specialrule{1pt}{0pt}{0pt}
$137A$ & $\left(0^4,12,34,15+36\right)$ & $\langle e_7\rangle$\\
\specialrule{1pt}{0pt}{0pt}
$137A_1$ & $\left(0^4,13+24,14-23,15+26\right)$ & $\langle e_7\rangle$
\\
\specialrule{1pt}{0pt}{0pt}
$137B$ & $\left(0^4,12,34,15+24+36\right)$ & $\langle e_7\rangle$\\
\specialrule{1pt}{0pt}{0pt}
$137B_1$ & $\left(0^4,13+24,14-23,15+26+34\right)$ & $\langle e_7\rangle$ \\
\specialrule{1pt}{0pt}{0pt}
$137C$ & $\left(0^4,12,14+23,16-35\right)$ & $\langle e_7\rangle$\\
\specialrule{1pt}{0pt}{0pt}
$137D$ & $\left(0^4,12,14+23,16+24-35\right)$ & $\langle e_7\rangle$\\
\specialrule{1pt}{0pt}{0pt}
$147A$ & $\left(0^3,12,13,0,16+25+34\right)$ & $\langle e_7\rangle$\\
\specialrule{1pt}{0pt}{0pt}
$147A_1$ & $\left(0^3,12,13,0,16+24+35\right)$ & $\langle e_7\rangle$\\
\specialrule{1pt}{0pt}{0pt}
$147B$ & $\left(0^3,12,13,0,14+26+35\right)$ & $\langle e_7\rangle$\\
\specialrule{1pt}{0pt}{0pt}
$147D$ & $\left(0^3,12,23,-13,15+16+26+2\cdot 34\right)$ & $\langle e_7\rangle$\\
\specialrule{1pt}{0pt}{0pt}
$147E(\lambda)$ & $\left(0^3,12,23,-13,-15+\lambda\cdot 26+(1-\lambda)\cdot 34\right)$, $\lambda\neq0,1$ & $\langle e_7\rangle$\\
\specialrule{1pt}{0pt}{0pt}
$147E_1(\lambda)$ & $\left(0^3,12,23,-13,-\lambda\cdot 16+\lambda\cdot 25+2\cdot 26-2\cdot 34\right)$, $\lambda>1$ & $\langle e_7\rangle$\\
\specialrule{1pt}{0pt}{0pt}
$157$ & $\left(0^2,12,0,0,0,13+24+56\right)$ & $\langle e_7\rangle$\\
\specialrule{1pt}{0pt}{0pt}
$247A$ & $\left(0^3,12,13,14,15\right)$ & $\langle e_6,e_7\rangle$\\
\specialrule{1pt}{0pt}{0pt}
$247B$ & $\left(0^3,12,13,14,35\right)$ & $\langle e_6,e_7\rangle$\\
\specialrule{1pt}{0pt}{0pt}
$247C$ & $\left(0^3,12,13,14+35,15\right)$ & $\langle e_6,e_7\rangle$\\
\specialrule{1pt}{0pt}{0pt}
$247D$ & $\left(0^3,12,13,14,25+34\right)$ & $\langle e_6,e_7\rangle$ \\
\specialrule{1pt}{0pt}{0pt}
$247E$ & $\left(0^3,12,13,14+15,25+34\right)$ & $\langle e_6,e_7\rangle$ \\
\specialrule{1pt}{0pt}{0pt}
$247E_1$ & $\left(0^3,12,13,14,24+35\right)$ & $\langle e_6,e_7\rangle$ \\
\specialrule{1pt}{0pt}{0pt}
$247F$ & $\left(0^3,12,13,24+35,25+34\right)$ & $\langle e_6,e_7\rangle$ \\
\specialrule{1pt}{0pt}{0pt}
$247F_1$ & $\left(0^3,12,13,24-35,25+34\right)$ & $\langle e_6,e_7\rangle$ \\
\specialrule{1pt}{0pt}{0pt}
$247G$ & $\left(0^3,12,13,14+15+24+35,25+34\right)$ & $\langle e_6,e_7\rangle$ \\
\specialrule{1pt}{0pt}{0pt}
$247H$ & $\left(0^3,12,13,14+24+35,25+34\right)$ & $\langle e_6,e_7\rangle$ \\
\specialrule{1pt}{0pt}{0pt}
$247H_1$ & $\left(0^3,12,13,14+24-35,25+34\right)$ & $\langle e_6,e_7\rangle$ \\
\specialrule{1pt}{0pt}{0pt}
$247I$ & $\left(0^3,12,13,25+34,35\right)$ & $\langle e_6,e_7\rangle$ \\
\specialrule{1pt}{0pt}{0pt}
$247J$ & $\left(0^3,12,13,15+35,25+34\right)$ & $\langle e_6,e_7\rangle$ \\
\specialrule{1pt}{0pt}{0pt}
$247K$ & $\left(0^3,12,13,14+35,25+34\right)$ & $\langle e_6,e_7\rangle$ \\
\specialrule{1pt}{0pt}{0pt}
$247L$ & $\left(0^3,12,13,14+23,15\right)$ & $\langle e_6,e_7\rangle$ \\
\specialrule{1pt}{0pt}{0pt}
$247M$ & $\left(0^3,12,13,14+23,35\right)$ & $\langle e_6,e_7\rangle$ \\
\specialrule{1pt}{0pt}{0pt}
$247N$ & $\left(0^3,12,13,15+24,23\right)$ & $\langle e_6,e_7\rangle$ \\
\specialrule{1pt}{0pt}{0pt}
$247O$ & $\left(0^3,12,13,14+35,15+23\right)$ & $\langle e_6,e_7\rangle$\\
\specialrule{1pt}{0pt}{0pt}
$247P$ & $\left(0^3,12,13,23,25+34\right)$ & $\langle e_6,e_7\rangle$\\
\bottomrule[1pt]
\end{tabu}}
\end{table}
}


{\footnotesize 
\begin{table}[H]
\centering
\caption{7-dimensional indecomposable 3-step nilpotent Lie algebras $-$ 2}\label{7d-ind-3step-2}
\vspace{0.25 cm}
{\tabulinesep=1.2mm
\begin{tabu}{l|c|c}
\toprule[1.5pt]
NLA & structure equations & center\\
\specialrule{1pt}{0pt}{0pt}
$247P_1$ & $\left(0^3,12,13,23,24+35\right)$ & $\langle e_6,e_7\rangle$ \\
\specialrule{1pt}{0pt}{0pt}
$247Q$ & $\left(0^3,12,13,14+23,25+34\right)$ & $\langle e_6,e_7\rangle$\\
\specialrule{1pt}{0pt}{0pt}
$247R$ & $\left(0^3,12,13,14+15+23,25+34\right)$ & $\langle e_6,e_7\rangle$\\
\specialrule{1pt}{0pt}{0pt}
$247R_1$ & $\left(0^3,12,13,14+23,24+35\right)$ & $\langle e_6,e_7\rangle$\\
\specialrule{1pt}{0pt}{0pt}
$257A$ & $\left(0^2,12,0,0,13+24,15\right)$ & $\langle e_6,e_7\rangle$\\
\specialrule{1pt}{0pt}{0pt}
$257B$ & $\left(0^2,12,0,0,13,14+25\right)$ & $\langle e_6,e_7\rangle$\\
\specialrule{1pt}{0pt}{0pt}
$257C$ & $\left(0^2,12,0,0,13+24,25\right)$ & $\langle e_6,e_7\rangle$ \\
\specialrule{1pt}{0pt}{0pt}
$257D$ & $\left(0^2,12,0,0,13+24,14+25\right)$ & $\langle e_6,e_7\rangle$ \\
\specialrule{1pt}{0pt}{0pt}
$257E$ & $\left(0^2,12,0,0,13+45,24\right)$ & $\langle e_6,e_7\rangle$ \\
\specialrule{1pt}{0pt}{0pt}
$257F$ & $\left(0^2,12,0,0,23+45,24\right)$ & $\langle e_6,e_7\rangle$ \\
\specialrule{1pt}{0pt}{0pt}
$257G$ & $\left(0^2,12,0,0,13+45,15+24\right)$ & $\langle e_6,e_7\rangle$ \\
\specialrule{1pt}{0pt}{0pt}
$257H$ & $\left(0^2,12,0,0,13+24,45\right)$ & $\langle e_6,e_7\rangle$ \\
\specialrule{1pt}{0pt}{0pt}
$257I$ & $\left(0^2,12,0,0,13+14,15+23\right)$ & $\langle e_6,e_7\rangle$ \\
\specialrule{1pt}{0pt}{0pt}
$257J$ & $\left(0^2,12,0,0,13+24,15+23\right)$ & $\langle e_6,e_7\rangle$ \\
\specialrule{1pt}{0pt}{0pt}
$257J_1$ & $\left(0^2,12,0,0,13+14+25,15+23\right)$ & $\langle e_6,e_7\rangle$ \\
\specialrule{1pt}{0pt}{0pt}
$257K$ & $\left(0^2,12,0,0,13,23+45\right)$ & $\langle e_6,e_7\rangle$ \\
\specialrule{1pt}{0pt}{0pt}
$257L$ & $\left(0^2,12,0,0,13+24,23+45\right)$ & $\langle e_6,e_7\rangle$ \\
\specialrule{1pt}{0pt}{0pt}
$357A$ & $\left(0^2,12,0,13,24,14\right)$ & $\langle e_5,e_6,e_7\rangle$\\
\specialrule{1pt}{0pt}{0pt}
$357B$ & $\left(0^2,12,0,13,23,14\right)$ & $\langle e_5,e_6,e_7\rangle$ \\
\specialrule{1pt}{0pt}{0pt}
$357C$ & $\left(0^2,12,0,13+24,23,14\right)$ & $\langle e_5,e_6,e_7\rangle$ \\
\bottomrule[1pt]
\end{tabu}}
\end{table}
}

{\footnotesize 
\begin{table}[H]
\centering
\caption{7-dimensional indecomposable 4-step nilpotent Lie algebras $-$ 1}\label{7d-ind-4step-1}
\vspace{0.25 cm}
{\tabulinesep=1.2mm
\begin{tabu}{l|c|c}
\toprule[1.5pt]
NLA & structure equations & center\\
\specialrule{1pt}{0pt}{0pt}
$1357A$ & $\left(0^3,12,14+23,0,15+26-34\right)$ & $\langle e_7\rangle$
\\
\specialrule{1pt}{0pt}{0pt}
$1357B$ & $\left(0^3,12,14+23,0,15-34+36\right)$ & $\langle e_7\rangle$\\
\specialrule{1pt}{0pt}{0pt}
$1357C$ & $\left(0^3,12,14+23,0,15+24-34+36\right)$ & $\langle e_7\rangle$\\
\specialrule{1pt}{0pt}{0pt}
$1357D$ & $\left(0^2,12,0,23,24,16+25+34\right)$ & $\langle e_7\rangle$\\
\specialrule{1pt}{0pt}{0pt}
$1357E$ & $\left(0^2,12,0,23,24,25+46\right)$ & $\langle e_7\rangle$\\
\specialrule{1pt}{0pt}{0pt}
$1357F$ & $\left(0^2,12,0,23,24,13+25-46\right)$ & $\langle e_7\rangle$\\
\bottomrule[1pt]
\end{tabu}}
\end{table}
}

{\footnotesize 
\begin{table}[H]
\centering
\caption{7-dimensional indecomposable 4-step nilpotent Lie algebras $-$ 2}\label{7d-ind-4step-2}
\vspace{0.25 cm}
{\tabulinesep=1.2mm
\begin{tabu}{l|c|c}
\toprule[1.5pt]
NLA & structure equations & center\\
\specialrule{1pt}{0pt}{0pt}
$1357F_1$ & $\left(0^2,12,0,23,24,13+25+46\right)$ & $\langle e_7\rangle$\\
\specialrule{1pt}{0pt}{0pt}
$1357G$ & $\left(0^2,12,0,23,14,16+25\right)$ & $\langle e_7\rangle$\\
\specialrule{1pt}{0pt}{0pt}
$1357H$ & $\left(0^2,12,0,23,14,16+25+26-34\right)$ & $\langle e_7\rangle$\\
\specialrule{1pt}{0pt}{0pt}
$1357I$ & $\left(0^2,12,0,23,14,25+46\right)$ & $\langle e_7\rangle$\\
\specialrule{1pt}{0pt}{0pt}
$1357J$ & $\left(0^2,12,0,23,14,13+25+46\right)$ & $\langle e_7\rangle$\\
\specialrule{1pt}{0pt}{0pt}
$1357L$ & $\left(0^2,12,0,13+24,14,15+23+\frac{1}{2}\cdot26+\frac{1}{2}\cdot34\right)$ & $\langle e_7\rangle$\\
\specialrule{1pt}{0pt}{0pt}
$1357M(\lambda)$ & $\left(0^2,12,0,13+24,14,15+\lambda\cdot26+(1-\lambda)\cdot34\right)$, $\lambda\neq 0$ & $\langle e_7\rangle$\\
\specialrule{1pt}{0pt}{0pt}
$1357N(\lambda)$ & $\left(0^2,12,0,13+24,14,15+\lambda\cdot23+34+46\right)$ & $\langle e_7\rangle$\\
\specialrule{1pt}{0pt}{0pt}
$1357O$ & $\left(0^2,12,0,13+24,23,16+25\right)$ & $\langle e_7\rangle$\\
\specialrule{1pt}{0pt}{0pt}
$1357P$ & $\left(0^2,12,0,13+24,23,15+26+34\right)$ & $\langle e_7\rangle$\\
\specialrule{1pt}{0pt}{0pt}
$1357P_1$ & $\left(0^2,12,0,13+24,23,15-26+34\right)$ & $\langle e_7\rangle$\\
\specialrule{1pt}{0pt}{0pt}
$1357Q$ & $\left(0^2,12,0,13,23+24,15+26\right)$ & $\langle e_7\rangle$\\
\specialrule{1pt}{0pt}{0pt}
$1357Q_1$ & $\left(0^2,12,0,13,23+24,15-26\right)$ & $\langle e_7\rangle$\\
\specialrule{1pt}{0pt}{0pt}
$1357QRS_1(\lambda)$ & $\left(0^2,12,0,13+24,14-23,15+\lambda\cdot26+(1-\lambda)\cdot34\right)$, $\lambda\neq 0$ & $\langle e_7\rangle$\\
\specialrule{1pt}{0pt}{0pt}
$1357R$ & $\left(0^2,12,0,13,23+24,16+25+34\right)$ & $\langle e_7\rangle$\\
\specialrule{1pt}{0pt}{0pt}
$1357S(\lambda)$ & $\left(0^2,12,0,13,23+24,15+16+25+\lambda\cdot 26+34\right)$, $\lambda\neq 1$ & $\langle e_7\rangle$\\
\specialrule{1pt}{0pt}{0pt}
$1457A$ & $\left(0^2,12,13,0,0,14+56\right)$ & $\langle e_7\rangle$\\
\specialrule{1pt}{0pt}{0pt}
$1457B$ & $\left(0^2,12,13,0,0,14+23+56\right)$ & $\langle e_7\rangle$\\
\specialrule{1pt}{0pt}{0pt}
$2357A$ & $\left(0^3,12,14+23,23,15-34\right)$ & $\langle e_6,e_7\rangle$\\
\specialrule{1pt}{0pt}{0pt}
$2357B$ & $\left(0^3,12,14+23,13,15-34\right)$ & $\langle e_6,e_7\rangle$\\
\specialrule{1pt}{0pt}{0pt}
$2357C$ & $\left(0^3,12,14+23,24,15-34\right)$ & $\langle e_6,e_7\rangle$\\
\specialrule{1pt}{0pt}{0pt}
$2357D$ & $\left(0^3,12,14+23,13+24,15-34\right)$ & $\langle e_6,e_7\rangle$\\
\specialrule{1pt}{0pt}{0pt}
$2357D_1$ & $\left(0^3,12,14+23,13-24,15-34\right)$ & $\langle e_6,e_7\rangle$\\
\specialrule{1pt}{0pt}{0pt}
$2457A$ & $\left(0^2,12,13,0,14,15\right)$ & $\langle e_6,e_7\rangle$\\
\specialrule{1pt}{0pt}{0pt}
$2457B$ & $\left(0^2,12,13,0,25,14\right)$ & $\langle e_6,e_7\rangle$\\
\specialrule{1pt}{0pt}{0pt}
$2457C$ & $\left(0^2,12,13,0,14+25,15\right)$ & $\langle e_6,e_7\rangle$\\
\specialrule{1pt}{0pt}{0pt}
$2457D$ & $\left(0^2,12,13,0,14+23+25,15\right)$ & $\langle e_6,e_7\rangle$\\
\specialrule{1pt}{0pt}{0pt}
$2457E$ & $\left(0^2,12,13,0,23+25,14\right)$ & $\langle e_6,e_7\rangle$\\
\specialrule{1pt}{0pt}{0pt}
$2457F$ & $\left(0^2,12,13,0,14+23,15\right)$ & $\langle e_6,e_7\rangle$\\
\specialrule{1pt}{0pt}{0pt}
$2457G$ & $\left(0^2,12,13,0,15+23,14\right)$ & $\langle e_6,e_7\rangle$\\
\specialrule{1pt}{0pt}{0pt}
$2457H$ & $\left(0^2,12,13,0,23,14+25\right)$ & $\langle e_6,e_7\rangle$\\
\specialrule{1pt}{0pt}{0pt}
$2457I$ & $\left(0^2,12,13,0,14+23,25\right)$ & $\langle e_6,e_7\rangle$\\
\bottomrule[1pt]
\end{tabu}}
\end{table}
}

{\footnotesize 
\begin{table}[H]
\centering
\caption{7-dimensional indecomposable 4-step nilpotent Lie algebras $-$ 3}\label{7d-ind-4step-3}
\vspace{0.25 cm}
{\tabulinesep=1.2mm
\begin{tabu}{l|c|c}
\toprule[1.5pt]
NLA & structure equations & center\\
\specialrule{1pt}{0pt}{0pt}
$2457I$ & $\left(0^2,12,13,0,14+23,25\right)$ & $\langle e_6,e_7\rangle$\\
\specialrule{1pt}{0pt}{0pt}
$2457J$ & $\left(0^2,12,13,0,14+23,23+25\right)$ & $\langle e_6,e_7\rangle$\\
\specialrule{1pt}{0pt}{0pt}
$2457K$ & $\left(0^2,12,13,0,15+23,14+25\right)$ & $\langle e_6,e_7\rangle$\\
\specialrule{1pt}{0pt}{0pt}
$2457L$ & $\left(0^2,12,13,23,14+25,15+24\right)$ & $\langle e_6,e_7\rangle$\\
\specialrule{1pt}{0pt}{0pt}
$2457L_1$ & $\left(0^2,12,13,23,14-25,15+24\right)$ & $\langle e_6,e_7\rangle$\\
\specialrule{1pt}{0pt}{0pt}
$2457M$ & $\left(0^2,12,13,23,15+24,14\right)$ & $\langle e_6,e_7\rangle$\\
\bottomrule[1pt]
\end{tabu}}
\end{table}
}

\bibliographystyle{plain}
\bibliography{bibliografia} 

\end{document}